\documentclass[fleqn,11pt]{article}
\usepackage{tikz}
\usetikzlibrary{arrows,shapes,chains}
\usepackage{graphics}
\usepackage{color}
\usepackage{mathrsfs}
\usepackage{amssymb}
\usepackage{amsmath}
\usepackage{amsfonts,amscd,amsthm}
\usepackage{bbm}
\usepackage{graphicx}
\usepackage{enumerate}
\usepackage{geometry}
\usepackage{xcolor}
\usepackage{txfonts}
\usepackage{cite}
\usepackage{bbding}
\DeclareMathOperator*\argmin{argmin}

\begin{document}

\pagenumbering{arabic}
\newcounter{comp1}

\newtheorem{definition}{Definition}[section]
\newtheorem{proposition}{Proposition}[section]
\newtheorem{example}{Example}[section]
\newtheorem{method}{Method}[section]
\newtheorem{lemma}{Lemma}[section]
\newtheorem{theorem}{Theorem}[section]
\newtheorem{corollary}{Corollary}[section]
\newtheorem{application}{Application}[section]
\newtheorem{assumption}{Assumption}
\newtheorem{algorithm}{Algorithm}[section]
\newtheorem{remark}{Remark}[section]
\newcommand{\fig}[1]{\begin{figure}[hbt]
                  \vspace{1cm}
                  \begin{center}
                  \begin{picture}(15,10)(0,0)
                  \put(0,0){\line(1,0){15}}
                  \put(0,0){\line(0,1){10}}
                  \put(15,0){\line(0,1){10}}
                  \put(0,10){\line(1,0){15}}
                  \end{picture}
                  \end{center}
                  \vspace{.3cm}
                  \caption{#1}
                  \vspace{.5cm}
                  \end{figure}}
\newcommand{\Axk}{A(x^k)}
\newcommand{\Aumb}{\sum_{i=1}^{N}A_{i}u_{i}-b}
\newcommand{\Kk}{K^k}
\newcommand{\Kki}{K_{i}^{k}}
\newcommand{\Aukmb}{\sum_{i=1}^{N}A_{i}u_{i}^{k}-b}
\newcommand{\Au}{\sum_{i=1}^{N}A_{i}u_{i}}
\newcommand{\Aukpmb}{\sum_{i=1}^{N}A_{i}u_{i}^{k+1}-b}
\newcommand{\nab}{\nabla^2 f(x^k)}
\newcommand{\xk}{x^k}
\newcommand{\ubk}{\overline{u}^k}
\newcommand{\uhk}{\hat u^k}
\newcommand{\B}{{\mathbb{B}}}
\newcommand{\N}{{\mathbb{N}}}
\newcommand{\R}{{\mathbb{R}}}
\newcommand{\Sph}{{\mathbb{S}}}
\def\QEDclosed{\mbox{\rule[0pt]{1.3ex}{1.3ex}}} 
\def\QEDopen{{\setlength{\fboxsep}{0pt}\setlength{\fboxrule}{0.2pt}\fbox{\rule[0pt]{0pt}{1.3ex}\rule[0pt]{1.3ex}{0pt}}}}
\def\QED{\QEDopen} 
\def\proof{\par\noindent{\em Proof.}}
\def\endproof{\hfill $\Box$ \vskip 0.4cm}
\newcommand{\RR}{\mathbf R}

\title {\bf
Level-set Subdifferential Error Bounds and Linear Convergence of Variable Bregman Proximal Gradient Method}
\author{Daoli Zhu\thanks {Antai College of Economics and Management and Sino-US Global Logistics
Institute, Shanghai Jiao Tong University, Shanghai, China({\tt
dlzhu@sjtu.edu.cn})}, Sien Deng\thanks {Department of Mathematical Sciences, Northern Illinois University, DeKalb, IL, USA({\tt sdeng@niu.edu})}, Minghua Li\thanks {School of Mathematics and Big Data, Chongqing University of Arts and Sciences, Yongchuan, Chongqing, China({\tt minghuali20021848@163.com})}, Lei Zhao\thanks {School of Naval Architecture, Ocean and Civil Engineering, Shanghai
Jiao Tong University, 200030 Shanghai, China({\tt l.zhao@sjtu.edu.cn})}}

\maketitle

\begin{abstract}
In this work,
we develop a level-set subdifferential error bound condition aiming towards convergence rate analysis of a variable Bregman proximal gradient (VBPG) method for a broad class of nonsmooth and nonconvex optimization problems.
It is proved that the aforementioned condition guarantees linear convergence of VBPG, and is weaker than Kurdyka-{\L}ojasiewicz property, weak metric subregularity and Bregman proximal error bound. Along the way, we are able to derive a number of  verifiable conditions for level-set subdifferential  error bounds to hold, and  necessary conditions and sufficient conditions for
linear convergence relative to a level set for nonsmooth and nonconvex optimization problems.
The newly established  results not only enable us to show that any accumulation point of the sequence generated by VBPG is at least a critical point of the limiting subdifferential or even a critical point of the proximal subdifferential with a fixed Bregman function in each iteration, but  also provide a fresh perspective that allows us to explore   inner-connections among many known sufficient conditions for linear convergence of various first-order methods.

\vspace{0.3cm}

%

\end{abstract}
\normalsize
\vspace{1cm}
\section{Introduction}\label{intro}
This paper studies the following nonconvex and nonsmooth optimization problem:
\begin{equation}\label{Prob:general-function}
\mbox{{\rm(P)}}\qquad\min_{x\in\RR^n}\qquad F(x)=f(x)+g(x)
\end{equation}
where $f:\RR^n\rightarrow(-\infty,\infty]$ is a proper lower semi-continuous (l.s.c) function that is smooth in $\mathbf{dom}f$,  and $g:~\RR^n\rightarrow(-\infty, \infty]$ is a proper l.s.c  function. We say that (P) is a convex problem (a fully nonconvex problem) if
both $f$ and $g$ are convex (both $f$ and $g$ are nonconvex).

Problem (P) arises naturally  in diverse areas such as compressed sensing~\cite{CandesTao05,Donoho06}, machine learning and statistics~\cite{Tibshirani1996}, principal component analysis~\cite{Aybat2018} and principal component pursuit~\cite{Aybat2014}. Typically these problems are of large scale. As the number of decision variables is huge, first-order methods and their enhanced versions are viewed to be a practical way to solve (P) ~\cite{LionsMercier1979,Cohen80,Nesterov13}.

By incorporating a Newton-like approach, we propose to solve (P) by a  variable Bregman proximal gradient (VBPG) method first introduced in \cite{Cohen80} with the name of Auxiliary Problem Principal (APP) method. An iteration of the  method   takes the form:
\begin{equation}\label{eq:sub-problem}
\mbox{{\rm (AP$_k$)}}\qquad x^{k+1}\in\argmin_{x\in\RR^n}\bigg{\{}\langle\nabla f(x^k),x-x^k\rangle+g(x)+\frac{1}{\epsilon^k}D^k(x^k,x)\bigg{\}},
\end{equation}
where $D^k$ is a variable Bregman distance  (see Section 2.1 for the definition of a Bregman distance). A nonsmooth version is investigated in \cite{CohenZ}. The classical proximal gradient (PG) method  is $D^k(x,y)=\frac{1}{2}\|x-y\|^2$.
 The second-order information through $D^k$ can be used to enhance the rate of  convergence  of the method \cite{Bonettini2016,Chouzenous2014}.
Some other choices of $D^k$ can be found in \cite{Banjac18}. Moreover, the VBPG method can  be combined with extrapolation, proximal alternating linearization and line search process \cite{ZhuMarcotte95}. VBPG can be also viewed as a forward-backward splitting method to find a critical point of (P): $x^{k+1}=(\Gamma^k+A)^{-1}(\Gamma^k-B)x^k$ with $A=\nabla f(x)$, $B=\partial_Pg(x)$, $\Gamma^k=\nabla K^k(x)/\epsilon^k$ (for notation of $K^k(x)$, $\epsilon^k$ and $\partial_Pg(x)$, see Section~\ref{sec:pre} for details).

Theory of error bounds (EB) has long been known playing an important role in
optimization theory~\cite{Rockafellar,Mordukhovich}, and a central role in  convergence analysis and convergence rate analysis  of various iterative  methods~\cite{Pan97}.
As we are interested in finding an optimal solution, or a critical point, or an optimal value for (P),  it is  natural to look at the following  types of   error bounds: the first type EB is an inequality that bounds the {\em distance from a set of test points to
a target set} (e.g., critical-point set of (P), optimal solution set of (P), or a level set of $F$) by
a {\em residual function}; while the second type  EB is an inequality that bounds certain {\em absolute values of the difference between function  $F$ values at a set of test points and  a target value} (e.g., a critical value of $F$, or the optimal value of (P)) by a {\em residual function}.
Prominent examples of first type error bounds include
  \cite{Hoffman1952,Cro78,BuF93,BuD02,Rob81}.
Pioneering  contributions to second type error bounds include~\cite{Pol63} and {\L}ojasiewicz inequality~\cite{Loj63}.


 When (P) is a convex problem, PG methods
exhibit  sublinear convergence rates~\cite{Bonettini2016,Nesterov13} and
achieve a  linear convergence rate ~\cite{Cohen17} if   $f$ is strongly convex.
Without strong convexity,~\cite{Necoara2018} examines
sufficient conditions for linear convergence of PG and acceleration techniques.

 Recently there is  a surge of interest in developing some first type error bound (EB) conditions that guarantee linear convergence for PG methods \cite{Lewis2018,LuoTseng92}, and in applying a generalized second type EB is ( Kurdyka-{\L}ojasiewicz (K{\L}) property)  to  obtain linear convergence of PG methods as well as a variety of other optimization methods~\cite{Attouch13,Chouzenous2014,Frankel2015,Bolte2010,LiPong18}. \cite{Schmidt2016} proposes a proximal-PL inequality that leads to an elegant linear convergence rate analysis for sequences of function-values generated by the PG method. We remark that the proximal-PL inequality condition combines and extends an idea originated from  metric functions for variational inequalities (VI) by reformulating a VI as a constrained continuous differentiable optimization problem through certain gap functions, see~\cite{Zhu94}.
In addition,
there are two major lines of research on  error bound conditions to achieve linear convergence guarantee for gradient descent methods. The first line of research is to find  connections among existing error bound conditions.
Examples of such work include \cite{Schmidt2016,Lewis2018,Zhang2019}.
Another  line of  research is the study of (P) when (P) is fully nonconvex. A sample of such work can be
found in \cite{Lewis2018,Ye18}.

Motivated  by the aforementioned works for a quest for linear convergence of PG methods, we are led to ask the following basic question:
{\it What are fundamental properties associated with
$F$ itself so that linear  convergence  of  VBPG is guaranteed?} This question leads us to look into  error bounds involving  level sets, subdifferentials  and  various  level-set  error bounds. A significant departure of our work to the above cited works is the use of  level sets as target sets 
 to establish error bound conditions whereas the above cited works  typically use optimal solution sets or sets of critical points (in the  nonconvex case) as target sets to establish error bound conditions.
In this work,
we have discovered a number of interesting  results  on level sets of
$F$,  revealed the roles of  level-set based  error bounds in achieving   linear convergence of VBPG, and uncovered  interconnections among level-set based  error bounds and other known error bounds  in  the literature.

This is a simplified version of the manuscript entitled ``An variational approach on level sets
and linear convergence of variable Bregman proximal gradient method for nonconvex optimization problems"\cite{ZhuDeng19}.
By introducing and examing
 a level-set sudifferential error bound condition carefully,  we are able to derive  linear convergence of VBPG under this condition.  
Some interesting features of this condition are as follows:
\begin{itemize}
\item[(i)] In the fully nonconvex setting (i.e., both $f$ and $g$ are nonconvex), this condition  is sufficient  for $Q-$linear convergence of $\{F(x^k)\}$ and $R-$linear convergence of $\{x^k\}$ generated by VBPG. Moreover, all known sufficient conditions  for linear convergence
of PG methods  imply this condition;

\item[(ii)] The level-set subdifferential EB condition along with  associated theorems  provides a unique perspective that allows us to make  connections with many known conditions in the literature which are shown to guarantee the linear convergence of PG.
\end{itemize}

In addition to the above contributions, we also provide necessary conditions and sufficient conditions for linear convergence with respect to level sets for VBPG. By examples of Subsection~\ref{subsec:examples}, we have shown that the notion of the level-set subdifferential EB condition is weaker than that of the K{\L} property, that of  weak metric subregularity, and that of Bregman proximal error bound. To our knowledge, this  is the first comprehensive work on  convergence rate analysis of VBPG. Moreover, a number of new results obtained in this work for VBPG are also new results even for PG methods.

The rest of this paper is structured as follows. Section~\ref{sec:pre} provides notation and preliminaries. Section~\ref{sec:convergence} presents the results on  convergence and linear convergence analysis for VBPG. Section~\ref{sec:LSEB} introduces level-set analysis,  studies level-set type error bounds, and provides the necessary and sufficient condition of linear convergence of VBPG under level-set based error bounds. Section \ref{sec:LSEB relationships} investigates connections of various level-set  error bounds established in this work with  existing error bounds.
Section~\ref{sec:LSEB_sufficiential} lists known sufficient conditions to guarantee the existence of level-set subdifferential error bounds. Finally,  we supply  Figure 1 in Section~\ref{sec:LSEB relationships} and Figure 2 in Section~\ref{sec:LSEB_sufficiential} to aid the reader to see easily inner relationships of these conditions and results.
\section{Notations and preliminaries}\label{sec:pre}
Throughout this paper, $\langle\cdot,\cdot\rangle$ and $\|\cdot\|$ denote the Euclidean scalar product of $\RR^n$ and its corresponding norm respectively. Let $\mathbf{C}$ be a subset of $\RR^n$ and $x$ be any point in $\RR^n$. Define
$$dist(x,\mathbf{C})=\inf\{\|x-z\|:z\in\mathbf{C}\}.$$
When $\mathbf{C}=\emptyset$, we set $dist(x,\mathbf{C})=\infty$.

The definitions we will use throughout the paper  are  standard in variational analysis (\cite{Rockafellar} and \cite{Mordukhovich}).
\begin{definition}[\cite{Rockafellar}]
Let $\psi$: $\RR^n\rightarrow\RR\cup\{+\infty\}$ be a proper lsc function.
\begin{itemize}
\item[{\rm(i)}] For each $\bar{x}\in\mathbf{dom}~\psi$, the Fr\'echet subdifferential of $\psi$ at $\bar{x}$, written $\partial_{F}\psi(\bar{x})$, is the set of vectors $\xi\in\RR^n$, which satisfy
    $$\liminf_{\substack{x\neq\bar{x}\\ x\rightarrow\bar{x}}}\frac{1}{\|x-\bar{x}\|}[\psi(x)-\psi(\bar{x})-\langle\xi,x-\bar{x}\rangle]\geq0.$$
If $x\notin\mathbf{dom} \psi$, then $\partial_F\psi(x)=\emptyset$.
\item[{\rm(ii)}] The limiting-subdifferential (\cite{Mordukhovich}), or simply the subdifferential for short, of $\psi$ at $\bar{x}\in\mathbf{dom}~\psi$, written $\partial_L\psi(\bar{x})$, is defined as follows:
    $$\partial_L\psi(\bar{x}):=\{\xi\in\RR^n:\exists x_n\rightarrow\bar{x}, \psi(x_n)\rightarrow \psi(\bar{x}), \xi_n\in\partial_F\psi(x_n)\rightarrow\xi\}.$$
\item[{\rm(iii)}] The proximal subdifferential of $\psi$ at $\bar{x}\in\mathbf{dom}\psi$ written $\partial_P\psi(\bar{x})$, is defined as follows:
    $$\partial_P\psi(\bar{x}):=\{\xi\in\RR^n:\exists\rho>0, \eta>0\; \mbox{s.t.}\; \psi(x)\geq\psi(\bar{x})+\langle\xi,x-\bar{x}\rangle-\rho\|x-\bar{x}\|^2 \forall x\in\mathbb{B}(\bar{x};\eta)\},$$
    where $\mathbb{B}(\bar{x}; \eta)$ is the open ball of radius $\eta>0$, centered at $\bar{x}$.
\end{itemize}
\end{definition}
\begin{definition}[\cite {Bernard05,Bolte2010,Rockafellar}]
Let $\psi:\RR^n\rightarrow(-\infty,\infty]$ be a proper lsc function.
\begin{itemize}
\item[{\rm(i)}]  (Definition~13.27 of \cite{Rockafellar}) A lsc function $\psi$ is said to be prox-regular at $\bar{x}\in \mathbf{dom}~\psi$ for subgradient $\bar{\nu}\in\partial_L\psi(\bar{x})$, if there exist parameters $\eta>0$ and $\rho\geq0$ such that for every point $(x,\nu)\in gph\partial_L\psi$ obeying $\|x-\bar{x}\|<\eta$, $|\psi(x)-\psi(\bar{x})|<\eta$ and $\|\nu-\bar{\nu}\|<\eta$ and $\nu\in\partial_L\psi(x)$, one has
    \[
    \psi(x')\geq\psi(x)+\langle\nu,x'-x\rangle-\frac{\rho}{2}\|x'-x\|^2\;\forall x'\in\mathbb{B}(\bar{x};\eta).
    \]
\item[{\rm(ii)}] (Proposition~3.3 of \cite{Bernard05})~ A lsc function $\psi$ is said to be uniformly prox-regular around $\bar{x}\in \mathbf{dom}~\psi$ , if there exist parameters $\eta>0$ and $\rho\geq0$ such that for every point $x, x'\in \mathbb{B}(\bar{x};\eta)$ and $\nu\in\partial_L\psi(x)$, one has
    \[
    \psi(x')\geq\psi(x)+\langle\nu,x'-x\rangle-\frac{\rho}{2}\|x'-x\|^2.
    \]
\item[{\rm(iii)}] (Definition~10 of \cite{Bolte2010})~A lsc function $\psi$ is semi-convex on $\mathbf{dom}~\psi$ with modulus $\rho>0$ if there exists a convex function $h:\RR^n\rightarrow\RR$ such that $\psi=h(x)-\frac{\rho}{2}\|x\|^2$.
\end{itemize}
\end{definition}
The following inclusions always hold: $\partial_P\psi(x)\subset\partial_{F}\psi(x)\subset\partial_L\psi(x)$.
If $\psi$ is uniformly prox-regular around $\bar{x}$ on $\mathbb{B}(\bar{x}; \eta)$ with $\eta>0$, we have $\partial_P\psi(x)=\partial_L\psi(x)$ for all $x\in\mathbb{B}(\overline{x}; \eta)$. In particular, $\partial_P\psi(x)=\partial_L\psi(x)$
 if $\psi$ is a semi-convex (convex) function.

Throughout the rest of  this paper, we make the following assumption on $f$ and $g$.
\begin{assumption}\label{assump1}
\begin{itemize}
\item[{\rm(i)}] $f:\mathbf{dom}\rightarrow (-\infty,\infty]$ is a differentiable function with $\mathbf{dom}~f$ convex and with gradient $L$-Lipschitz continuous.
\item[{\rm(ii)}] $g$ is proper lower semicontinuous on $\mathbf{dom}~g$, and $\mathbf{dom}~g$ is a convex set.
\item[{\rm(iii)}] $F$ is level-bounded i.e., the set $\{x\in\RR^n : F(x)\leq r\}$ is bounded (possibly empty) for every $r\in\RR$.
\end{itemize}
\end{assumption}
A few remarks about Assumption 1 are in order.
By Theorem 3.2.12 of \cite{Ortega}, the following descent property of $f$  holds
\begin{equation*}
\frac{L}{2}\|y-x\|^2+\langle\nabla f(x),y-x\rangle\geq f(y)-f(x)\;\forall x,y\in \mathbf{dom}f.
\end{equation*}
From {\rm{(i)} and {\rm{(ii)}}, $\mathbf{dom}~F$ is a convex set. As a consequence of {\rm{(iii)},  the optimal value $F^*$ of (P) is finite and the optimal solution set $\mathbf{X}^*$ of (P) is non-empty.

A vector $x$ satisfying $0\in\partial_P F(x)$ is called a proximal critical point. The set of all proximal critical points of $F$ is denoted by $\bar{\mathbf{X}}_P$. By Assumption~\ref{assump1}, $\bar{\mathbf{X}}_P\neq\emptyset$. The limiting critical point is defined as:
$$\bar{\mathbf{X}}_L:=\{x : 0\in\nabla f(x)+\partial_L g(x)\}.$$
In general $\bar{\mathbf{X}}_P\subseteq\bar{\mathbf{X}}_L$ and the equality holds  if $\partial_Pg(x)=\partial_Lg(x)$.
By Proposition 2.3 of \cite{Ye18}, $\partial_P F(x)=\nabla f(x)+\partial_P g(x)$.
\subsection{Variable Bregman distance, Bregman type mappings and functions}\label{sec:Variable Bregman}
Let a sequence of functions $\{K^k,k\in\mathbb{N}\}$ and positive numbers $\{\epsilon^k,k\in\mathbb{N}\}$ be given, where the function $K^k$  is strongly convex and differentiable with Lipschitz gradient. For each $k$, define a variable Bregman distance
\begin{equation}
D^k(x,y)=K^k(y)-[K^k(x)+\langle\nabla K^k(x),y-x\rangle].
\end{equation}
 The variable Bregman distance $D^k$ measures the proximity between two points $(x,y)$; that is,
$D^k(x,y)\geq0$ and  $D^k(x,y)=0$ if and only if  $x=y$.
We  make the following standing assumption on the functions $K^k(x)$.
\begin{assumption}\label{assump2}
\begin{itemize}
\item[{\rm(i)}] For each $k$, $K^k$ is strongly convex with uniformly modulus $m$ and with its gradient $\nabla K^k$ being uniformly $M$-Lipschitz.
\item[{\rm(ii)}] The parameter $\epsilon^k$ satisfies: $0<\underline{\epsilon}\leq\epsilon^k\leq\overline{\epsilon}$.
\end{itemize}
\end{assumption}
Under Assumption~2, $\{D^k~|~k\in N\}$ uniformly satisfies:
\begin{eqnarray*}
&&m\|x-y\|^2\leq\langle\nabla_xD^k(x,y),x-y\rangle\leq M\|x-y\|^2,\\
&&m\|x-y\|^2\leq\langle\nabla_yD^k(x,y),y-x\rangle\leq M\|x-y\|^2,\\
&&\frac{m}{2}\|x-y\|^2\leq D^k(x,y)\leq\frac{M}{2}\|x-y\|^2.
\end{eqnarray*}
 To simplify our analysis, in what follows, we will drop the sub-index $k$. Thanks to Assumption~2, the results we
will establish
hold for all $k$. To this end,
let  a strongly twice differentiable convex function $K$  along with a positive $\epsilon\in
 (\underline{\epsilon}, \overline{\epsilon})$ be given.
Suppose a Bregman distance $D$ is constructed based on $K$. We introduce
 the following Bregman type mappings and functions which will play a key role for the convergence analysis of the VBPG method.\\
{\bf Bregman Proximal Envelope Function} $E_{D,\epsilon}$ is defined by
\begin{equation}
E_{D,\epsilon}(x)=\min_{y\in\RR^n}\{f(x)+\langle\nabla f(x),y-x\rangle+g(y)+\frac{1}{\epsilon}D(x,y)\}\;\forall x\in\RR^n,
\end{equation}
which is expressed as the value function of optimization problem (AP$^k$) (see~\eqref{eq:sub-problem}), where $x^k$ is replaced by $x$.\\
{\bf Bregman Proximal Mapping} $T_{D,\epsilon}$ is defined by
\begin{equation}\label{defi:Tk}
T_{D,\epsilon}(x)=\argmin_{y\in\RR^n}\langle\nabla f(x),y-x\rangle+g(y)+\frac{1}{\epsilon}D(x,y)\;\forall x\in\RR^n,
\end{equation}
which can be viewed as the set of  optimizers of optimization problem (AP$^k$). By Assumptions~\ref{assump1} and~\ref{assump2}, $T_{D,\epsilon}(x)$ is non-empty, the mapping $T_{D,\epsilon}(x)$ could be multi-valued.\\
{\bf Bregman proximal gap function} $G_{D,\epsilon}$ is defined by
\begin{eqnarray}
G_{D,\epsilon}(x)=-\frac{1}{\epsilon}\min_{y\in\RR^n}\{\langle\nabla f(x),y-x\rangle+g(y)-g(x)+\frac{1}{\epsilon}D(x,y)\}\;\forall x\in\RR^n.
\end{eqnarray}
Obviously, we have $G_{D,\epsilon}(x)\geq0$ for all $x$. If $g$ is semi-convex, the following optimization problem is equivalent to the differential inclusion problem $0\in\partial_P F(x)$ associated with problem (P) (see Proposition~\ref{prop:Gk})
\[
\min_{x\in\RR^n}G_{D,\epsilon}(x).
\]
The
above mappings and functions enjoy some favorable properties summarized in the following:
\begin{proposition}{\bf(Global properties of Bregman type mappings and functions)}\label{prop:Ek} Let a  Bregman function $D$ be given.
Suppose that Assumptions~\ref{assump1}  and \ref{assump2} hold, and that $\epsilon\in (0,m/L)$. Then for any $x\in\RR^n$, $t_{D,\epsilon}(x)\in T_{D,\epsilon}(x)$, we have
{\rm
\begin{itemize}
\item[(i)] $E_{D,\epsilon}(x)=F(x)-\epsilon G_{D,\epsilon}(x)$;
\item[(ii)] $F\big{(}t_{D,\epsilon}(x)\big{)}\leq E_{D,\epsilon}(x)-a\|x-t_{D,\epsilon}(x)\|^2$ with $a=\frac{1}{2}(\frac{m}{\epsilon}-L)$;
\item[(iii)] $F\big{(}t_{D,\epsilon}(x)\big{)}\leq F(x)-a\|x-t_{D,\epsilon}(x)\|^2.$
\end{itemize}
}
\end{proposition}
\begin{proof}
{\rm (i):} This follows immediately from the definitions $G_{D,\epsilon}(x)$ and $E_{D,\epsilon}(x)$.\\
{\rm (ii) \& (iii):} Since $\nabla f$ is $L$-Lipschitz, one has
\begin{eqnarray}
E_{D,\epsilon}(x)&=&f(x)+\langle\nabla f(x),t_{D,\epsilon}(x)-x\rangle+g\big{(}t_{D,\epsilon}(x)\big{)}+\frac{1}{\epsilon}D\big{(}x,t_{D,\epsilon}(x)\big{)}\nonumber\\
      &\geq&f\big{(}t_{D,\epsilon}(x)\big{)}-\frac{L}{2}\|x-t_{D,\epsilon}(x)\|^2+g\big{(}t_{D,\epsilon}(x)\big{)}+\frac{1}{\epsilon}D\big{(}x,t_{D,\epsilon}(x)\big{)}.\nonumber
\end{eqnarray}
Thus, by (i) and the fact $D\big{(}x,t_{D,\epsilon}(x)\big{)}\geq\frac{m}{2}\|x-t_{D,\epsilon}(x)\|^2$, we get
\begin{eqnarray*}
F\big{(}t_{D,\epsilon}(x)\big{)}&\leq& E_{D,\epsilon}(x)-\frac{1}{\epsilon}D\big{(}x,t_{D,\epsilon}(x)\big{)}+\frac{L}{2}\|x-t_{D,\epsilon}(x)\|^2\nonumber\\
                     &\leq& F(x)-\frac{1}{2}(\frac{m}{\epsilon}-L)\|x-t_{D,\epsilon}(x)\|^2.\nonumber
\end{eqnarray*}
This completes the proof. \end{proof}

\begin{proposition}\label{prop:Fk}{\bf (Properties of $\partial_P F$)}
Suppose that Assumptions~\ref{assump1} and \ref{assump2} hold. Then for all $t_{D,\epsilon}(x)\in T_{D,\epsilon}(x)$ we have
\begin{itemize}
\item[{\rm(i)}] $dist\bigg{(}0,\partial_PF\big{(}t_{D,\epsilon}(x)\big{)}\bigg{)}\leq(L+\frac{M}{\underline{\epsilon}})\|x-t_{D,\epsilon}(x)\|$;
\item[{\rm(ii)}] If $x\in T_{D,\epsilon}(x)$, then $0\in\partial_P F(x)$.
\end{itemize}
\end{proposition}
\begin{proof}
{\rm(i):} The optimality condition of optimizer $t_{D,\epsilon}(x)\in T_{D,\epsilon}(x)$ yields
\[
0\in\nabla f(x)+\partial_P g\big{(}t_{D,\epsilon}(x)\big{)}+\frac{1}{\epsilon}\nabla_yD\big{(}x,t_{D,\epsilon}(x)\big{)}.
\]
Let $\xi=\nabla f\big{(}t_{D,\epsilon}(x)\big{)}-\nabla f(x)-\frac{1}{\epsilon}\nabla_yD\big{(}x,t_{D,\epsilon}(x)\big{)}$. Then we have
\[
\xi\in\partial_P F\big{(}t_{D,\epsilon}(x)\big{)}=\nabla f(t_{D,\epsilon}(x))+\partial_P g(t_{D,\epsilon}(x)).
\]
By Assumptions~\ref{assump1} and~\ref{assump2}, we have
\begin{eqnarray}
\|\xi\|\leq\|\nabla f\big{(}t_{D,\epsilon}(x)\big{)}-\nabla f(x)\|+\frac{1}{\underline{\epsilon}}\|\nabla_yD\big{(}x,t_{D,\epsilon}(x)\big{)}\|\leq(L+\frac{M}{\underline{\epsilon}})\|x-t_{D,\epsilon}(x)\|,\nonumber
\end{eqnarray}
which follows the desired statement.\\
{\rm(ii):} The claim follows directly from statement (i).
\end{proof}
Under assumptions of Proposition~\ref{prop:Ek}, if $\overline{\epsilon}<\frac{m}{L}$, then it's easy to show that functions $E_{D,\epsilon}(x)$ and $G_{D,\epsilon}(x)$ are continuous, mapping $T_{D,\epsilon}(x)$ is closed and is continuous whenever $T_{D,\epsilon}(x)$ is single valued (see Proposition 6.1 of~\cite{ZhuDeng19}).

Before the end of this section, we introduce the following lemma about the generalized descent inequality.
\begin{lemma}[Generalized descent inequality in the nonconvex case]\label{lemma:1}
Suppose that Assumptions~\ref{assump1} and~\ref{assump2} hold. For any $t_{D,\epsilon}(x)\in T_{D,\epsilon}(x)$, $x\in\RR^n$, $u\in\RR^n$, we have that
\begin{eqnarray}\label{eq:descent}
\mathfrak{a}\left[F(t_{D,\epsilon}(x))-F(u)\right]\leq \mathfrak{b}\|u-x\|^2-\|u-t_{D,\epsilon}(x)\|^2-\mathfrak{c}\|x-t_{D,\epsilon}(x)\|^2,
\end{eqnarray}
where $\mathfrak{a}=2$, $\mathfrak{b}=\frac{M}{\underline{\epsilon}}+2+3L$ and $\mathfrak{c}=\frac{m}{\overline{\epsilon}}-(L+2)$.
\end{lemma}
\begin{proof}
Denote $\Delta=\langle\nabla f(x),t_{D,\epsilon}(x)-u\rangle+g(t_{D,\epsilon}(x))-g(u)$. First, we estimate the lower bound of $\Delta$:\\
\begin{eqnarray}\label{eq:bound1-f-n}
\Delta&=&\langle\nabla f(x),t_{D,\epsilon}(x)-u\rangle+g(t_{D,\epsilon}(x))-g(u)\nonumber\\
&=&\langle\nabla f(x),t_{D,\epsilon}(x)-x\rangle+\langle\nabla f(x),x-u\rangle+g(t_{D,\epsilon}(x))-g(u)\nonumber\\
&\geq&f\left(t_{D,\epsilon}(x)\right)-f(x)-\frac{L}{2}\|x-t_{D,\epsilon}(x)\|^2+\langle\nabla f(x),x-u\rangle+g(t_{D,\epsilon}(x))-g(u)\nonumber\\
&&\qquad\qquad\qquad\qquad\qquad\qquad\mbox{(since $f$ is gradient Lipschitz with modulus $L$)}\nonumber\\
&=&F\left(t_{D,\epsilon}(x)\right)-F(u)-\frac{L}{2}\|x-t_{D,\epsilon}(x)\|^2+\underbrace{f(u)-f(x)-\langle\nabla f(x),u-x\rangle}_{\delta_1}.
\end{eqnarray}
By the gradient Lipschitz continuity of $f$, we estimate the term $\delta_1$ in~\eqref{eq:bound1-f-n}:
\begin{eqnarray}
\delta_1&=&f(u)-f(x)-\langle\nabla f(u),u-x\rangle+\langle\nabla f(x)-\nabla f(u),x-u\rangle\nonumber\\
&\geq&-\frac{L}{2}\|u-x\|^2-\|\nabla f(x)-\nabla f(u)\|\cdot\|x-u\|\nonumber\\
&\geq&-\frac{3L}{2}\|u-x\|^2.\nonumber
\end{eqnarray}
Therefore, we have that
\begin{equation}\label{eq:bound1}
\Delta\geq F\left(t_{D,\epsilon}(x)\right)-F(u)-\frac{L}{2}\|x-t_{D,\epsilon}(x)\|^2-\frac{3L}{2}\|u-x\|^2.
\end{equation}
Since $t_{D,\epsilon}(x)$ solves the minimization problem~\eqref{defi:Tk}, we have
\begin{eqnarray}\label{eq:optimalcondition-gn1}
\Delta&=&\langle\nabla f(x),t_{D,\epsilon}(x)-u\rangle+g(t_{D,\epsilon}(x))-g(u)\nonumber\\
&\leq&\frac{1}{\epsilon}\left[D(x,u)-D(x,t_{D,\epsilon}(x))\right]\nonumber\\
&\leq&\frac{M}{2\underline{\epsilon}}\|u-x\|^2-\frac{m}{2\overline{\epsilon}}\|x-t_{D,\epsilon}(x)\|^2.\qquad\mbox{(by Assumption~\ref{assump2})}
\end{eqnarray}
Since $-\frac{1}{2}\|u-t_{D,\epsilon}(x)\|^2+\|u-x\|^2+\|x-t_{D,\epsilon}(x)\|^2\geq0$,~\eqref{eq:optimalcondition-gn1} follows that
\begin{eqnarray}\label{eq:optimalcondition-gn11}
\Delta&\leq&\frac{M}{2\underline{\epsilon}}\|u-x\|^2-\frac{1}{2}\|u-t_{D,\epsilon}(x)\|^2+\|u-x\|^2+\|x-t_{D,\epsilon}(x)\|^2-\frac{m}{2\overline{\epsilon}}\|x-t_{D,\epsilon}(x)\|^2\nonumber\\
&\leq&\left(\frac{M}{2\underline{\epsilon}}+1\right)\|u-x\|^2-\frac{1}{2}\|u-t_{D,\epsilon}(x)\|^2-\frac{m-2\overline{\epsilon}}{2\overline{\epsilon}}\|x-t_{D,\epsilon}(x)\|^2.
\end{eqnarray}
The desired result follows by
combing~\eqref{eq:bound1} and~\eqref{eq:optimalcondition-gn11}.
\end{proof}
\begin{remark}[Cost-to-go inequality {[46]}] From this lemma with $\kappa=\max\{\frac{2\mathfrak{b}-1}{\mathfrak{a}},\frac{2\mathfrak{b}-\mathfrak{c}}{\mathfrak{a}}\}>0$, we also get for $x,u\in\RR^n$,
\begin{eqnarray}
F\left(t_{D,\epsilon}(x)\right)-F(u)&\leq&\frac{1}{\mathfrak{a}}\left\{2\mathfrak{b}\|u-t_{D,\epsilon}(x)\|^2+2\mathfrak{b}\|t_{D,\epsilon}(x)-x\|^2-\|u-t_{D,\epsilon}(x)\|^2-\mathfrak{c}\|x-t_{D,\epsilon}(x)\|^2\right\}\nonumber\\
&\leq&\kappa\left(\|u-t_{D,\epsilon}(x)\|^2+\|x-t_{D,\epsilon}(x)\|^2\right)
\end{eqnarray}
which is one cost-to-go estimate [46].
\end{remark}
\subsection{The properties of Bregman type mapping and function under semiconvexity of $g$}
\begin{proposition}\label{prop:singlevalue} {\bf (Single-valueness of Bregman proximal mappings)}
Suppose that Assumptions~\ref{assump1} and~\ref{assump2} hold, and that  $g$ is semiconvex on $\RR^n$ with constant $\rho$ and $\overline{\epsilon}<\min\{\frac{m}{L},\frac{m}{\rho}\}$. Then for all $x\in\RR^n$, $T_{D,\epsilon}(x)$ is single-valued.
\end{proposition}
\begin{proof} The claim is derived directly by the definition of semi-convexity.\quad
\end{proof}
\begin{proposition}\label{prop:Gk}{\bf (Further properties of Bregman type mappings and functions)}
Suppose that the assumptions of Proposition~\ref{prop:singlevalue} hold. Then for $x\in\RR^n$ and $\overline{\epsilon}<\min\{\frac{m}{L},\frac{m}{\rho}\}$ the following statements hold:
\begin{itemize}
\item[{\rm(i)}] $E_{D,\epsilon}(x)\leq F(x)-\frac{1}{2}\big{(}\frac{m}{\overline{\epsilon}}-\rho\big{)}\|x-T_{D,\epsilon}(x)\|^2$;
\item[{\rm(ii)}] $\frac{1}{2\overline{\epsilon}^2}(m-\overline{\epsilon}\rho)\|x-T_{D,\epsilon}(x)\|^2\leq G_{D,\epsilon}(x)$;
\item[{\rm(iii)}] $G_{D,\epsilon}(x)\leq\frac{\overline{\epsilon}}{2\underline{\epsilon}(m-\overline{\epsilon}\rho)}dist^2\left(0,\partial_P F(x)\right)$;
\item[{\rm(iv)}] $\|x-T_{D,\epsilon}(x)\|\leq\left(\frac{\overline{\epsilon}}{m-\overline{\epsilon}\rho}\right)\sqrt{\frac{\overline{\epsilon}}{\underline{\epsilon}}}dist\left(0,\partial_PF(x)\right)$;
\item[{\rm(v)}] $G_{D,\epsilon}(x)=0$ if only if $x=T_{D,\epsilon}(x)$ or $0\in\partial_PF(x)$.
\end{itemize}
\end{proposition}
\begin{proof}
By Proposition~\ref{prop:singlevalue}, $T_{D,\epsilon}(x)$ is single-valued.\\
{\rm (i):} The optimality condition for the minimization problem in~\eqref{eq:sub-problem} follows that
\begin{equation}\label{eq:13}
0\in\nabla f(x)+\partial_L g\big{(}T_{D,\epsilon}(x)\big{)}+\frac{1}{\epsilon}\nabla_{y}D\big{(}x,T_{D,\epsilon}(x)\big{)}.
\end{equation}
Since $g$ is l.s.c and semiconvex with $\rho$, we have
\begin{eqnarray}\label{eq:10}
g(x)&\geq&g\big{(}T_{D,\epsilon}(x)\big{)}-\frac{\rho}{2}\|x-T_{D,\epsilon}(x)\|^2-\langle\nabla f(x)+\frac{1}{\epsilon}\nabla_{y}D\big{(}x,T_{D,\epsilon}(x)\big{)},x-T_{D,\epsilon}(x)\rangle\nonumber\\
&\geq&g\big{(}T_{D,\epsilon}(x)\big{)}-\frac{\rho}{2}\|x-T_{D,\epsilon}(x)\|^2-\langle\nabla f(x),x-T_{D,\epsilon}(x)\rangle\nonumber\\
&&+\frac{1}{\epsilon}D(x,T_{D,\epsilon}(x))-\frac{1}{\epsilon}D(x,x)+\frac{m}{2\overline{\epsilon}}\|x-T_{D,\epsilon}(x)\|^2\qquad\mbox{(by Assumption~\ref{assump2})}\nonumber\\
    &\geq&g\big{(}T_{D,\epsilon}(x)\big{)}-\langle\nabla f(x),x-T_{D,\epsilon}(x)\rangle+\frac{1}{\epsilon}D(x,T_{D,\epsilon}(x))+\frac{1}{2}\left(\frac{m}{\overline{\epsilon}}-\rho\right)\|x-T_{D,\epsilon}(x)\|^2.\nonumber
\end{eqnarray}
Adding $f(x)$ to both sides, by the definition of $E_{D,\epsilon}(x)$ the claim is provided.\\
{\rm(ii):} From statement (i) of Proposition~\ref{prop:Ek} we have
\begin{eqnarray*}
G_{D,\epsilon}(x)=\frac{1}{\epsilon}\big{(}F(x)-E_{D,\epsilon}(x)\big{)}\geq\frac{1}{2\overline{\epsilon}^2}\big{(}m-\overline{\epsilon}\rho\big{)}\|x-T_{D,\epsilon}(x)\|^2.\quad\mbox{(by (i) of this proposition)}
\end{eqnarray*}
{\rm(iii):} For $\overline{\epsilon}<\min\{\frac{m}{L},\frac{m}{\rho}\}$, we have
\begin{eqnarray*}
\epsilon G_{D,\epsilon}(x)=-\langle\nabla f(x), T_{D,\epsilon}(x)-x\rangle-g\left(T_{D,\epsilon}(x)\right)+g(x)-\frac{1}{\epsilon}D(x,T_{D,\epsilon}(x)).
\end{eqnarray*}
Let $\nu\in\partial_Pg(x)$, thanks the semiconvex of $g$, we get
\begin{eqnarray}
\epsilon G_{D,\epsilon}(x)&\leq&-\langle\nabla f(x), T_{D,\epsilon}(x)-x\rangle-\langle\nu, T_{D,\epsilon}(x)-x\rangle+\frac{\rho}{2}\|x-T_{D,\epsilon}(x)\|^2-\frac{m}{2\overline{\epsilon}}\|x-T_{D,\epsilon}(x)\|^2\nonumber\\
&=&-\langle\nabla f(x)+\nu, T_{D,\epsilon}(x)-x\rangle-\frac{1}{2}\left(\frac{m}{\overline{\epsilon}}-\rho\right)\|x-T_{D,\epsilon}(x)\|^2\nonumber\\
&\leq&\|\nabla f(x)+\nu\|\cdot\|x-T_{D,\epsilon}(x)\|-\frac{1}{2}\left(\frac{m}{\overline{\epsilon}}-\rho\right)\|x-T_{D,\epsilon}(x)\|^2\nonumber\\
&\leq&\frac{\overline{\epsilon}}{2(m-\overline{\epsilon}\rho)}\|\nabla f(x)+\nu\|^2.
\end{eqnarray}
Therefore $G_{D,\epsilon}(x)\leq\frac{\overline{\epsilon}}{2\underline{\epsilon}(m-\overline{\epsilon}\rho)}\|\nabla f(x)+\nu\|^2$, $\forall\nu\in\partial_Pg(x)$, and the claim is verified.\\
{\rm(iv)} and {\rm(v):} The statement (iv) is a simple consequence of (ii) and (iii). (v) follows directly from statements (ii), (iii) and~\eqref{eq:13}.
\end{proof}
\section{Convergence analysis of VBPG}\label{sec:convergence}
Section \ref{sec:convergence} studies convergence behaviors of sequences generated by the update formula of VBPG~\eqref{eq:sub-problem}.
We will  use $D^k$ {\it explicitly} and assume that $D^k$ and parameters $\epsilon^k$ satisfy  Assumption~\ref{assump2} uniformly throughout  Sections \ref{sec:convergence}.
A number of basic properties of sequences $\{x^k\}$ and $\{F(x^k)\}$ are summarized in the following proposition.
\begin{proposition}\label{proposition4} Suppose that the Assumptions~\ref{assump1} and~\ref{assump2} hold and $\overline{\epsilon}<\frac{m}{L}$. Let $\{x^k\}$ be a sequence generated by the VBPG method. Then the following assertions hold:
\begin{itemize}
\item[{\rm(i)}] The sequence $\{F(x^k)\}$ is strictly decreasing (unless $x^k\in\bar{\mathbf{X}}_P$ for some $k$).
\item[{\rm(ii)}] $\{x^k\}$ is bounded, and any cluster point $\bar{x}$ of $\{x^k\}$ is a limiting critical point of $F$: $0\in\partial_L F(\bar{x})$. Furthermore, $\bar{x}$ is also actually a proximal critical point of $F$: $0\in\partial_P F(\bar{x})$.
\item[{\rm(iii)}] Let $\Omega$ be the set of accumulation points of $\{x^k\}$. Then $F_{\zeta}:=\lim\limits_{k\rightarrow+\infty} F(x^k)$ exists and $F(\bar{x})\leq F_{\zeta}$ for every $\bar{x}$ in $\Omega$.
\end{itemize}
\end{proposition}
\begin{proof}
{\rm (i):} By (iii) of Proposition~\ref{prop:Ek}, we have
\begin{equation}\label{eq:22}
F(x^{k+1})\leq F(x^k)-a\|x^k-x^{k+1}\|^2\quad\mbox{with}\quad a=\frac{1}{2}(\frac{m}{\epsilon}-L).
\end{equation}
If $x^k=x^{k+1}$, then by (ii) of Proposition~\ref{prop:Fk}, we have $x^k\in\bar{\mathbf{X}}_P$. Otherwise $\{F(x^k)\}$ is strictly decreasing.\\
{\rm (ii):} By summation for~\eqref{eq:22}, we have
\[
a\sum_{k=0}^{N}\|x^k-x^{k+1}\|^2\leq F(x^0)-F(x^{N+1})\leq F(x^0)-F^*\;\forall N.
\]
and it follows that
\[
\sum\limits_{k=0}^{\infty}\|x^k-x^{k+1}\|^2<+\infty,\quad\|x^k-x^{k+1}\|\rightarrow 0,\;\mbox{when}\; k\rightarrow\infty.
\]
The boundedness of $\{x^k\}$ comes from Assumption~\ref{assump1}, $F=(f+g)$ is level bounded along with the fact that $\{F(x^k)\}$ is strictly decreasing. Thus $\{x^k\}$ has at least one cluster point. Let $\bar{x}$ denote such a point and $x^{k'}\rightarrow\bar{x}$, $k'\rightarrow\infty$. From statement (i) of Proposition~\ref{prop:Fk}, we have
$$dist\left(0,\partial_LF(x^{k'})\right)\leq dist\left(0,\partial_PF(x^{k'})\right)\leq(L+\frac{M}{\underline{\epsilon}})\|x^{k'}-x^{k'-1}\|.$$
Thus $dist\left(0,\partial_LF(x^{k'})\right)\rightarrow0$ as $k'\rightarrow\infty$. By the closedness of $\partial_L F(\cdot)$, we have $0\in \partial_L F(\bar{x})$.\\
Furthermore, by the update formula of VBPG in (AP$_k$), we get
\begin{eqnarray}
&&\langle\nabla f(x^k),x^{k+1}-x^{k}\rangle+g(x^{k+1})+\frac{1}{\epsilon^k}D^k(x^k,x^{k+1})\nonumber\\
&\leq&\langle\nabla f(x^k),x-x^k\rangle+g(x)+\frac{1}{\epsilon^k}D^k(x^k,x)\nonumber\\
&\leq&f(x)-f(x^k)+\frac{L}{2}\|x-x^k\|^2+g(x)+\frac{1}{\epsilon^k}D^k(x^k,x)\;\forall x\in dom(F).\nonumber
\end{eqnarray}
By Assumption~\ref{assump2} for $D^k$, the above inequality yields
\begin{eqnarray}\label{eq:eq}
&&\langle\nabla f(x^k),x^{k+1}-x\rangle+g(x^{k+1})+\frac{m}{2\overline{\epsilon}}\|x^k-x^{k+1}\|^2\nonumber\\
&\leq&f(x)-f(x^k)+\frac{L}{2}\|x-x^k\|^2+g(x)+\frac{M}{2\underline{\epsilon}}\|x^k-x\|^2\;\forall x\in dom(F).
\end{eqnarray}
Taking $k=k'$ such that $\lim\limits_{k'\rightarrow\infty}x^{k'}=\bar{x}$, from the continuity of $f$ and lower semicontinuous of $g$, we obtain that $\lim\limits_{k'\rightarrow\infty}f(x^{k'})=f(\bar{x})$ and $\lim\limits_{k'\rightarrow\infty}g(x^{k'})\geq g(\bar{x})$. Therefore, by taking $k'\rightarrow\infty$ on both sides of~\eqref{eq:eq} and one has
$$F(\bar{x})\leq F(x)+(\frac{L}{2}+\frac{M}{2\underline{\epsilon}})\|x-\bar{x}\|^2\;\forall x\in dom(F),$$
which implies $0\in\partial_PF(\bar{x})$.\\
{\rm (iii):} By~\eqref{eq:22}, $\lim\limits_{k\rightarrow\infty}F(x^k)\rightarrow F_{\zeta}\geq F^*$. Let $\bar{x}\in\Omega$, then there exists a subsequence
$x^{k'}$ of $\{x^k\}$ such that  $x^{k'}\rightarrow\bar{x}$. By the lower semicontinuity of $F$ on \mbox{dom}~$F$ and the convergence  of $\{F (x^k)\}$,
we have $F(\bar{x})\leq\lim\limits_{k'\rightarrow\infty}F(x^{k'})=F_{\zeta}$.
\end{proof}
In order to study the linear convergence of the sequence generated by VBPG, we need the following concept of value proximity error bound. Let $\bar{x}\in \RR^n$ and $\bar{F}=F(\bar{x})$. For given positive numbers $\eta$ and $\mu$, let
$$\mathfrak{B}(\bar{x};\eta,\nu)=\mathbb{B}(\bar{x};\eta)\cap\{x\in \RR^n~:~\bar{F}<F(x)<\bar{F}+\nu\}.$$
\begin{definition}[VBPG iteration based Value Proximity Error Bound (VP-EB)] Let $\{x^k\}$ be the sequence generated by VBPG method and $\bar{x}$ be an accumulation point of $\{x^k\}$. We say the VP-EB holds at $\bar{x}$ if there exist $\kappa'$, $\eta$ and $\nu>0$ such that
\begin{equation}\label{eq:VP_EB}
F(x^{k+1})-F(\bar{x})\leq\kappa'\|x^k-x^{k+1}\|^2\;\forall x^{k+1}\in\mathfrak{B}(\bar{x};\eta,\nu).
\end{equation}
\end{definition}
We next show that a sequence generated by
 (\ref{eq:sub-problem}) is convergent and has a finite length property.
\begin{proposition}{\bf(Finite length property of whole sequence $\{x^k\}$)}\label{prop:3.1}
Let the sequence $\{x^k\}$ be generated by VBPG method and $\bar{x}$ be an accumulation point of $\{x^k\}$. Suppose that the Assumptions~\ref{assump1} and~\ref{assump2} hold and $\overline{\epsilon}<\frac{m}{L}$  and that the VP-EB holds at the point $\bar{x}$ with $\kappa'$, $\eta$ and $\nu>0$. Let $a$ be the constant given in Proposition~\ref{prop:Ek} and $\bar{F}=F(\bar{x})$. Then the following statements hold:
\begin{itemize}
\item[{\rm(i)}] $x^k\in\mathfrak{B}(\bar{x};\eta,\nu)$~~$\forall k\geq k_0$;
\item[{\rm(ii)}] $\sum\limits_{i=0}^{+\infty}\|x^i-x^{i+1}\|<+\infty$ (finite length property);
\item[{\rm(iii)}] The sequence $\{x^k\}$ actually converges to the point $\bar{x}$ which is a proximal critical point of $F$.
\end{itemize}
\end{proposition}
\begin{proof}
\rm{(i)}: By (i) and (iii) of Proposition \ref{proposition4}, there is $k_0$ such that we have $\bar{F}<F(x^k)<\bar{F}+\nu$ $\forall k\geq k_0$. From assumptions, without loss of generality, we assume that
\begin{eqnarray}
&&\bar{F}<F(x^{k_0})<\bar{F}+\nu\label{eq:condition1}\\
\mbox{and}&&\|x^{k_0}-\bar{x}\|+\frac{2(\sqrt{a}+\sqrt{\kappa'})}{a}\sqrt{F(x^{k_0})-\bar{F}}<\eta.\label{eq:condition2}
\end{eqnarray}
We will use the Principle of Mathematical Introduction to prove that the sequence $\{x^k\}\subset\mathfrak{B}(\bar{x};\eta,\nu)$ $\forall k\geq k_0$. It is clear that $x^{k_0}\in\mathfrak{B}(\bar{x};\eta,\nu)$ by (\ref{eq:condition1}) and (\ref{eq:condition2}). The inequalities $\bar{F}<F(x^{k_0+1})\leq F(x^{k_0})<\bar{F}+\nu$ hold trivially. On the other hand, by~\eqref{eq:22}, we have $$\|x^{k_0+1}-x^{k_0}\|\leq\sqrt{\frac{F(x^{k_0})-F(x^{k_0+1})}{a}}\leq\sqrt{\frac{F(x^{k_0})-\bar{F}}{a}}$$ and
$$\|x^{k_0+1}-\bar{x}\|\leq\|x^{k_0}-\bar{x}\|+\|x^{k_0}-x^{k_0+1}\|\leq\|x^{k_0}-\bar{x}\|+\sqrt{\frac{F(x^{k_0})-\bar{F}}{a}}<
\eta.~~(by~\eqref{eq:condition2})$$
Thus $x^{k_0+1}\in\mathfrak{B}(\bar{x};\eta,\nu)$. Now suppose that $x^i\in\mathfrak{B}(\bar{x};\eta,\nu)$ for $i=k_0+1,..,k_0+k$ and $x^{k_0+k}\neq x^{k_0+k+1}$. Note that $F(x^{k_0+1})>F(x^{k_0+2})>\cdots>F(x^{k_0+k})>F(x^{k_0+k+1})>\bar{F}$. We need to show that $x^{k_0+k+1}\in\mathfrak{B}(\bar{x};\eta,\nu)$. By the concavity of function $h(y)=y^{\frac{1}{2}}$, we have, for $i=k_0+1,k_0+2,\dots,k_0+k$, that
$$
\left(F(x^i)-\bar{F}\right)^{\frac{1}{2}}-\left(F(x^{i+1})-\bar{F}\right)^{\frac{1}{2}}
\geq \frac{1}{2}\frac{[F(x^i)-F(x^{i+1})]}{\left(F(x^i)-\bar{F}\right)^{\frac{1}{2}}}.$$
Recalling that $x^{i+1}\in T_{D^i,\epsilon^i}(x^i)$ and
applying (iii) of Proposition~\ref{prop:Ek}~and~\eqref{eq:VP_EB} of VP-EB to $[F(x^i)-F(x^{i+1})]$ and $(F(x^i)-\bar{F})^{1/2}$ respectively one has
$$\frac{2\sqrt{\kappa'}}{a}||x^i-x^{i-1}||[\left(F(x^i)-\bar{F}\right)^{\frac{1}{2}}-\left(F(x^{i+1})-\bar{F}\right)^{\frac{1}{2}}]\geq ||x^i-x^{i+1}||^2.$$
It follows from $2\sqrt{d_1 d_2}\leq d_1+d_2$ with nonnegative $d_1$ and $d_2$ that
\begin{eqnarray}\label{eq:34}
2\|x^{i+1}-x^i\|\leq\|x^{i}-x^{i-1}\|+\frac{2\sqrt{\kappa'}}{a}\left[\left(F(x^i)-\bar{F}\right)^{\frac{1}{2}}-\left(F(x^{i+1})-\bar{F}\right)^{\frac{1}{2}}\right].
\end{eqnarray}
Summing~\eqref{eq:34} for $i=k_0+1,...,k_0+k$, we obtain
\begin{eqnarray}\label{eq:51}
&&\sum_{i=k_0+1}^{k_0+k}\|x^{i+1}-x^i\|+\|x^{k_0+k+1}-x^{k_0+k}\|\nonumber\\
&\leq&\|x^{k_0+1}-x^{k_0}\|+\frac{2\sqrt{\kappa'}}{a}\left[\left(F(x^{k_0+1})-\bar{F}\right)^{\frac{1}{2}}-\left(F(x^{k_0+k+1})-\bar{F}\right)^{\frac{1}{2}}\right].
\end{eqnarray}
Using  \eqref{eq:51} along with the triangle inequality, we have
\begin{eqnarray*}
\|\bar{x}-x^{k_0+k+1}\|&\leq&\|\bar{x}-x^{k_0}\|+\|x^{k_0}-x^{k_0+1}\|+\sum_{i=k_0+1}^{k_0+k}\|x^{i+1}-x^i\|\\
                   &\leq&\|\bar{x}-x^{k_0}\|+2\|x^{k_0}-x^{k_0+1}\|+\frac{2\sqrt{\kappa'}}{a}\left[\left(F(x^{k_0+1})-\bar{F}\right)^{\frac{1}{2}}\right]\\
                   &\leq&\|\bar{x}-x^{k_0}\|+2\sqrt{\frac{F(x^{k_0})-\bar{F}}{a}}+\frac{2\sqrt{\kappa'}}{a}\left[\left(F(x^{k_0})-\bar{F}\right)^{\frac{1}{2}}\right]\\
&<&\eta.~~~(\mbox{by}~\eqref{eq:condition2})
\end{eqnarray*}
 This shows that $x^{k_0+k+1}\in\mathfrak{B}(\bar{x};\eta,\nu)$, and
(i) is proved by the Principle of Mathematical Induction.\\
\rm{(ii)} and \rm{(iii)}:
A direct consequence of \eqref{eq:51} is, for all $k$,
$$\sum_{i=k_0+1}^{k_0+k}\|x^{i+1}-x^i\|\leq\|x^{k_0+1}-x^{k_0}\|+\frac{2\sqrt{\kappa'}}{a}\left[\left(F(x^{k_0+1})-\bar{F}\right)^{\frac{1}{2}}\right]<+\infty.$$
Therefore, we have
$$\sum_{i=0}^{+\infty}\|x^{i+1}-x^i\|<+\infty.$$
In particular, this implies that the sequence $\{x^k\}$ is convergent and thus it actually converges to the point $\bar{x}$.  And $\bar{x}$ is a desired critical point of $F$ by Proposition~\ref{proposition4}.
\end{proof}
The main result of this section follows.
\begin{theorem}[Sufficient conditions for local linear convergence] \label{theo1} Suppose that all the conditions of Proposition \ref{prop:3.1} hold. Then $\{F(x^k)\}$ converges to value $\bar{F}$ at the $Q$-linear rate of convergence; that is, there are some $\beta\in(0,1)$ and $k_0$ such that
\begin{equation}\label{linearrate}
F(x^{k+1})-\bar{F}\leq\beta(F(x^k)-\bar{F})\;\forall k\geq k_0.
\end{equation}
Moreover, the sequence $\{x^k\}$  converges at the $R$-linear rate  to the critical point $\bar{x}$, which is either a limiting critical point or a proximal critical point of $F$.
\end{theorem}
\begin{proof}
In view of Proposition~\ref{prop:3.1}, there is $k_0$ such that for $k\geq k_0$ such that $\{x^k\}\subset\mathfrak{B}(\bar{x};\eta,\nu)$. It follows that
\begin{eqnarray}
F(x^{k+1})-\bar{F}&=&\left(F(x^k)-\bar{F}\right)+\left(F(x^{k+1})-F(x^k)\right)\nonumber\\
                    &\leq&\left(F(x^k)-\bar{F}\right)-a\|x^{k+1}-x^k\|^2\quad\mbox{(by (iii) in Proposition~\ref{prop:Ek})}\nonumber\\
                    &\leq&\left(F(x^k)-\bar{F}\right)-a\left(\frac{1}{\kappa'}\right)\left(F(x^{k+1})-\bar{F}\right).\quad\mbox{(by VP-EB condition)}
\end{eqnarray}
Therefore, one has
\begin{eqnarray}
F(x^{k+1})-\bar{F}\leq\frac{1}{1+a\left(\frac{1}{\kappa'}\right)}\left(F(x^k)-\bar{F}\right)\hspace{3mm}\forall k\geq k_0.
\end{eqnarray}
The above estimation shows that $\{F(x^k)\}$ converges  to $\bar{F}$ at the  Q-linear rate; that is,
\begin{equation}\label{Q-linear}
F(x^{k+1})-\bar{F}\leq \beta\left(F(x^k)-\bar{F}\right)\hspace{3mm}\forall k\geq k_0,
\end{equation}
where $\beta:=\frac{1}{1+a\left(\frac{1}{\kappa'}\right)}\in (0,1)$. We now derive the R-linear rate of convergence of  $\{x^k\}$. We have
\begin{eqnarray*}
\|x^k-x^{k+1}\|^2
&\leq&\frac{1}{a}\bigg{[}\big{(}F(x^k)-\bar{F}\big{)}-\big{(}F(x^{k+1})-\bar{F}\big{)}\bigg{]}\quad\mbox{(by~\eqref{eq:22})}\\
&\leq&\frac{1}{a}\big{(}F(x^k)-\bar{F}\big{)}\\
&\leq&\frac{\beta^{(k-k_0)}}{a}(F(x^{k_0})-\bar{F}).~~\mbox{ (by (\ref{Q-linear}))}
\end{eqnarray*}
From the above inequality, we see that
$$\|x^k-x^{k+1}\|\leq\hat{M}(\sqrt{\beta})^{(k-k_0)}~~~ \forall k> k_0,\quad\hat{M}:=\sqrt{\frac{F(x^{k_0})-\bar{F}}{a}}.$$
By Proposition~\ref{prop:3.1}, we have $\{x^k\}$ converges to the critical point $\bar{x}$ and
$\|x^k-\bar{x}\|\leq\sum_{i=k}^{\infty}\|x^i-x^{i+1}\|\leq\frac{\hat{M}}{1-\sqrt{\beta}}(\sqrt{\beta})^{(k-k_0)}$. This
shows that $\{x^k\}$ converges to the desired critical point $\bar{x}$ at the R-linear rate; that is, $$\limsup_{k\rightarrow\infty}\sqrt[(k-k_0)]{\|x^k-\bar{x}\|}=\sqrt{\beta}<1.\qquad\qquad\qquad $$
This completes the proof. \end{proof}
\section{Level-set based error bounds and necessary and sufficient conditions for linear convergence of VBPG}\label{sec:LSEB}
\subsection{Level-set subdifferential EB implies VP-EB}
In the rest of this paper, unless otherwise stated, we will always choose $\bar{F}=F(\bar{x})$ for some
given $\bar{x}\in{\bf dom}F$. Set $[F\leq\bar{F}]=\{x\in\RR^n:F(x)\leq F(\bar{x})\}$ and $[F>\bar{F}]=\{x\in\RR^n:F(x)>F(\bar{x})\}$. In this subsection, we examine level-set subdifferential  and level-set Bregman proximal error bounds.
\begin{definition}[Level-set error bounds]\label{LSEB}
\begin{itemize}
\item[(i)] ({\bf Level-set subdifferential error bound}) The  function $F$  is said to satisfy the level-set subdifferential error bound (EB) condition at $\bar{x}$ with exponent $\gamma>0$ if there exist $\eta>0$, $\nu>0$, and $c_1>0$ such that the following inequality holds:
\begin{equation}\label{LS-EB}
dist^{\gamma}(x,[F\leq\bar{F}])\leq c_1dist\big{(}0,\partial_P F(x)\big{)}~~\forall
x\in\mathfrak{B}(\bar{x};\eta,\nu).
\end{equation}
\item[(ii)]({\bf Level-set Bregman proximal error bound}) Given a Bregman function $D$ along with $\epsilon>0$, $F$ is said to satisfy the level-set Bregman proximal EB condition at $\bar{x}$ with exponent $p>0$ , if there exist $\eta>0$, $\nu>0$, and
$\theta>0$ such that the following inequality holds:
\begin{equation}\label{LB-EB}
dist^p(x,[F\leq\bar{F}])\leq\theta dist\left(x,T_{D,\epsilon}(x)\right)~~\forall x\in\mathfrak{B}(\bar{x};\eta,\nu).
\end{equation}
\end{itemize}
\end{definition}

\begin{proposition}[Level-set subdifferential EB implies level-set Bregman proximal EB]\label{theo:4.2}
Suppose Assumptions~\ref{assump1} and~\ref{assump2} hold with $\overline{\epsilon}<\frac{m}{L}$. Assume the level-set subdifferential EB holds at $\bar{x}$ with exponent $\gamma\in(0,\infty)$ over $\mathfrak{B}(\bar{x};\eta,\nu)$.
Then there are $N>\max\{\frac{2\overline{\epsilon}\nu}{m-\overline{\epsilon}L}/(\frac{\eta}{2})^2,1\}$ and $\theta>0$ such that
\begin{equation}
dist^p(x,[F\leq\bar{F}])\leq\theta dist\left(x,T_{D,\epsilon}(x)\right)\;\forall x\in\mathfrak{B}(\bar{x},\frac{\eta}{2},\frac{\nu}{N}),
\end{equation}
where $p=\frac{1}{\min\{\frac{1}{\gamma},1\}}$, $\theta_1=1+c_1^{\frac{1}{\gamma}}(L+\frac{M}{\underline{\epsilon}})^{\frac{1}{\gamma}}(\frac{\eta}{2})^{\frac{1}{\gamma}-1}$,
$\theta_2=(\frac{\eta}{2})^{1-\frac{1}{\gamma}}+c_1^{\frac{1}{\gamma}}(L+\frac{M}{\underline{\epsilon}})^{\frac{1}{\gamma}}$
 and
$\theta=\max\{\theta_1,\theta_2\}$.
\end{proposition}
\begin{proof}
By assumptions, $\overline{\epsilon}<m/L$, $T_{D,\epsilon}(x)\not=\emptyset$. For $x\in\mathfrak{B}(\bar{x};\frac{\eta}{2},\frac{\nu}{N})$, let $t_p(x)\in Proj_{T_{D,\epsilon}(x)}(x)$. If $F\big{(}t_{p}(x)\big{)}\leq\bar{F}$, then $t_{p}(x)\in[F\leq\bar{F}]$, and we have $$dist(x,[F\leq\bar{F}])\leq\|x-t_{p}(x)\|.$$
Now we consider the non-trivial case $F\big{(}t_{p}(x)\big{)}>\bar{F}$.
If  $x\in \mathfrak{B}(\bar{x};\frac{\eta}{2},\frac{\nu}{N})\subset\mathfrak{B}(\bar{x};\eta,\nu)$ with
$N$ satisfying the assumptions, then by (iii) of Proposition~\ref{prop:Ek}, we have that
$$\frac{1}{2}\left(\frac{m}{\overline{\epsilon}}-L\right)\|x-t_{p}(x)\|^2\leq F(x)-F\big{(}t_{p}(x)\big{)}\leq F(x)-\bar{F}\leq\frac{\nu}{N}.$$
Since $N>\frac{2\overline{\epsilon}\nu}{m-\overline{\epsilon}L}/(\frac{\eta}{2})^2$,  we get $\sqrt{\frac{2\nu\overline{\epsilon}}{N(m-\overline{\epsilon}L)}}<\frac{\eta}{2}$ and
$\|x-t_{p}(x)\|<\frac{\eta}{2}$.
As $\| x-\bar{x}\|\leq\frac{\eta}{2}$,
it follows that $\|t_{p}(x)-\bar{x}\|<\eta$, which yields $t_{p}(x)\in\mathbb{B}(\bar{x};\eta)$.
Hence for any $x\in\mathfrak{B}(\bar{x};\frac{\eta}{2},\frac{\nu}{N})$, by the level-set subdifferential EB we have
\[
dist(x,[F\leq\bar{F}])\leq\|x-t_{p}(x)\|+dist\big{(}t_{p}(x),[F\leq\bar{F}]\big{)}
=\|x-t_{p}(x)\|+c_1^{\frac{1}{\gamma}}dist^{\frac{1}{\gamma}}\big{(}0,\partial_PF\big{(}t_{p}(x)\big{)}\big{)}.
\]
Therefore,  for any $x\in\mathfrak{B}(\bar{x};\frac{\eta}{2},\frac{\nu}{N})$,
by (i) of Proposition~\ref{prop:Fk}, we have that
\begin{eqnarray*}
dist(x,[F\leq\bar{F}])\leq\|x-t_{p}(x)\|+
c_1^{\frac{1}{\gamma}}(L+\frac{M}{\underline{\epsilon}})^{\frac{1}{\gamma}}\|x-t_{p}(x)\|^{\frac{1}{\gamma}}.
\end{eqnarray*}
Since $\|x-t_{p}(x)\|<\frac{\eta}{2}$, by the above inequality, we have the following
estimate
\begin{eqnarray*}
    dist(x,[F\leq\bar{F}])&\leq&\left\{
\begin{array}{ll}
\theta_1\|x-t_{p}(x)\| &\mbox{if}\quad 0<\gamma\leq 1, \\
\theta_2\|x-t_{p}(x)\|^{\frac{1}{\gamma}} &\mbox{if}\quad \gamma>1 \\
\end{array}
\right.\\
&\leq&\theta\|x-t_{p}(x)\|^{\frac{1}{p}}=\theta dist^{\frac{1}{p}}\left(x, T_{D,\epsilon}(x)\right).
\end{eqnarray*}
This completes the proof.
\end{proof}
Next proposition shows that level-set subdifferential EB plays a key role for linear convergence of VBPG.
\begin{proposition}[Level-set subdifferential EB implies VP-EB]
Suppose that the Assumptions~\ref{assump1} and~\ref{assump2} hold. Let $\{x^k\}$ be the sequence generated by VBPG method and $\bar{x}$ be an accumulation point of $\{x^k\}$. Assume that the level-set subdifferential EB holds at the point $\bar{x}$ with exponent $\gamma\in(0,1]$, $\eta>0$ and $\nu>0$. Then there exists $k_0'>0$ such that VP-EB holds at $\bar{x}$ for all $k\geq k_0'$.
\end{proposition}
\begin{proof}
Let $x_p^{k+1}\in[F\leq\bar{F}]$ such that $\|x^{k+1}-x_p^{k+1}\|=dist(x^{k+1},[F\leq\bar{F}])$. From the cost-to-go inequality in Lemma 2.1 with $u=x_p^{k+1}$, we have
\begin{equation}\label{eq:1}
F(x^{k+1})-F(x_p^{k+1})\leq\kappa(\|x_p^{k+1}-x^{k+1}\|^2+\|x^k-x^{k+1}\|^2).
\end{equation}
The VBPG update process implies that there is $k_0'>0$ such that for $k\geq k_0'$, we have $\|x^k-x^{k+1}\|\leq1$. Since level-set subdifferential EB holds at $\bar{x}$ with $\gamma\in(0,1]$, $\eta,\nu>0$, for $x^{k+1}\in\mathfrak{B}(\bar{x};\eta,\nu)$ and $k\geq k_0'$, we have
\begin{eqnarray}\label{eq:2}
\|x^{k+1}-x_p^{k+1}\|^2&=&dist^2(x^{k+1},[F\leq\bar{F}])\nonumber\\
&\leq&(c_1)^{\frac{2}{\gamma}}dist^{\frac{2}{\gamma}}\left(0,\partial_PF(x^{k+1})\right)\nonumber\\
&\leq&(c_1)^{\frac{2}{\gamma}}(L+\frac{M}{\underline{\epsilon}})^{\frac{2}{\gamma}}\|x^k-x^{k+1}\|^{\frac{2}{\gamma}}\qquad\mbox{(by (i) of Proposition 2.2)}\nonumber\\
&\leq&(c_1)^{\frac{2}{\gamma}}(L+\frac{M}{\underline{\epsilon}})^{\frac{2}{\gamma}}\|x^k-x^{k+1}\|^2.\qquad\mbox{(since $\gamma\in(0,1]$, $\|x^k-x^{k+1}\|\leq1$)}
\end{eqnarray}
Using~\eqref{eq:1} and~\eqref{eq:2}, it yields
$$F(x^{k+1})-\bar{F}\leq F(x^{k+1})-F(x_p^{k+1})\leq\kappa'\|x^k-x^{k+1}\|^2\quad\mbox{($\kappa'=\max\{(c_1)^{\frac{2}{\gamma}}(L+\frac{M}{\underline{\epsilon}})^{\frac{2}{\gamma}}\kappa,\kappa\}$)}$$
which implies the claim.
\end{proof}
\subsection{The strong level-set error bounds and necessary and sufficient conditions for linear convergence}
We now turn to study the notion of the strong level-set error bounds holding on a  set $[\bar{F}< F\leq\bar{F}+\nu]$. This notion plays an important role in deriving
a sufficient condition and a necessary condition for linear convergence relative to level sets.
\begin{definition}[Strong level-set error bounds]
\begin{itemize}
\item[(i)]({\bf Strong level-set subdifferential error bound}) We say that $F$ satisfies the strong level-set subdifferential EB condition on $[\bar{F}<F<\bar{F}+\nu]$ with the value $\bar{F}$ and $\nu>0$ if there exists $c_1'>0$ such that
\[
dist(x,[F\leq\bar{F}])\leq c_1'dist\big{(}0,\partial_P F(x)\big{)}~~\forall x\in[\bar{F}<F<\bar{F}+\nu].
\]
\item[(ii)]({\bf Strong level-set Bregman proximal error bound}) Given a Bregman distance $D$ along with $\epsilon>0$, $F$ satisfies the strong level-set Bregman proximal EB condition on $[\bar{F}<F<\bar{F}+\nu]$ with $\bar{F}$ and $\nu>0$ if there exists $\theta'>0$ such that
\[
dist(x,[F\leq\bar{F}])\leq\theta'dist\left(x,T_{D,\epsilon}(x)\right)~~\forall x\in[\bar{F}<F<\bar{F}+\nu].
\]
\end{itemize}
\end{definition}
\begin{corollary}[Strong level-set subdifferential EB $\Rightarrow$ Strong level-set Bregman proximal EB]\label{Corollary:Strong}
Suppose Assumptions~\ref{assump1} and~\ref{assump2} hold with $\overline{\epsilon}<\frac{m}{L}$. Assume the strong level-set subdifferential EB holds over $[\bar{F}<F<\bar{F}+\nu]$. Then there is $\theta'=1+c_1'(L+\frac{M}{\underline{\epsilon}})>0$ such that
\[
dist(x,[F\leq\bar{F}])\leq\theta'dist\left(x,T_{D,\epsilon}(x)\right)~~\forall x\in[\bar{F}<F<\bar{F}+\nu].
\]
\end{corollary}
\begin{proof}
The claim is proved by the same argument for the proof of Proposition~\ref{theo:4.2} with $\gamma=1$.
\end{proof}
The following theorem gives a necessary condition and a sufficient condition for linear convergence relative to a level set.
\begin{theorem}({\bf Necessary and sufficient conditions for linear convergence relative to $[F\leq\bar{F}]$})\label{theo:n-s}
Suppose that Assumptions~\ref{assump1} and~\ref{assump2} hold. Let a sequence $\{x^k\}$ be generated by the VBPG method, let $\bar{x}$ be an accumulation point of $\{x^k\}$, and let $\nu>0$ be given.
\begin{itemize}
\item[{\rm(i)}] For any  initial point $x^0\in[\bar{F}<F<\bar{F}+\nu]$,  if the strong Bregman proximal EB condition holds  on $[\bar{F}<F<\bar{F}+\nu]$ with $\theta'\in\left(\sqrt{\frac{\mathfrak{c}}{\mathfrak{b}}},\sqrt{\frac{\mathfrak{c}}{\mathfrak{b}-1}}\right)$, then the VBPG method converges linearly respect to level-set $[F\leq\bar{F}]$, i.e.,
\begin{eqnarray}\label{x-Q-linear}
dist\left(x^{k+1},[F\leq\bar{F}]\right)\leq\beta dist\left(x^{k},[F\leq\bar{F}]\right)\quad \forall k\geq0,
\end{eqnarray}
with $\beta:=\sqrt{\mathfrak{b}-\frac{\mathfrak{c}}{(\theta')^2}}\in(0,1)$, where the values of $\mathfrak{b}$ and $\mathfrak{c}$ are appeared in Lemma~\ref{lemma:1}.
\item[{\rm(ii)}] If $g$ is semi-convex on $\RR^n$, $\overline{\epsilon}<\min\{\frac{m}{L},\frac{m}{\rho}\}$ and the VBPG method converges linearly in the sense of~\eqref{x-Q-linear} with $\beta\in (0,1)$, then $F$ satisfies the strong level-set subdifferential EB condition on $[\bar{F}<F<\bar{F}+\nu]$ with $c_1'=\frac{\overline{\epsilon}}{(1-\beta)(m-\overline{\epsilon}\rho)}\sqrt{\frac{\overline{\epsilon}}{\underline{\epsilon}}}$ and strong level-set Bregman proximal EB with $\theta'=1+c_1'(L+\frac{M}{\underline{\epsilon}})$.
\end{itemize}
\end{theorem}
\begin{proof}
(i) Since $\{F(x^k)\}$ is strictly decreasing and converges to $F_{\zeta}\geq\bar{F}$, then we must have $F(x^k)\geq\bar{F}$. Note that the equality implies $x^k\in\bar{x}_P$. Therefore, for given $x^0\in[\bar{F}<F<\bar{F}+\nu]$, we have $x^k\in[\bar{F}<F<\bar{F}+\nu]$. Let $x_p^k={Proj}_{[F\leq\bar{F}]}(x^k)$. Then $F(x_p^k)\leq\bar{F}$. By Lemma~\ref{lemma:1} with  $u=x_p^k$ in~\eqref{eq:descent}, we have $F(x^{k+1})\geq \bar{F}$ and
    \[
    0\leq\mathfrak{a}[F(x^{k+1})-F(x_p^k)]\leq\mathfrak{b}\|x_p^k-x^k\|^2-\|x_p^k-x^{k+1}\|^2-\mathfrak{c}\|x^k-x^{k+1}\|^2.
    \]
This together with strong Bregman proximal EB condition yields
    \[
    \|x_p^k-x^{k+1}\|^2\leq\mathfrak{b}\|x_p^k-x^k\|^2-\mathfrak{c}\|x^k-x^{k+1}\|^2\leq\mathfrak{b}\|x_p^k-x^k\|^2-\frac{\mathfrak{c}}{(\theta')^2}\|x_p^k-x^k\|^2.
    \]
Thus, one has
\[
    dist\left(x^{k+1},[F\leq\bar{F}]\right)\leq\|x_p^k-x^{k+1}\|\leq\left(\mathfrak{b}-\frac{\mathfrak{c}}{(\theta')^2}\right)^{\frac{1}{2}}dist\left(x^k,[F\leq\bar{F}]\right).
\]
(ii) By semi-convexity of $g$ and $\overline{\epsilon}<\min\{\frac{m}{L},\frac{m}{\rho}\}$, it follows from Proposition \ref{prop:singlevalue} that $T_{D,\epsilon}(x)$ is single-valued. Let $T_{D,\epsilon}(x)_p=Proj_{[F\leq\bar{F}]}\left(T_{D,\epsilon}(x)\right)$.  Then we see that
    \begin{eqnarray}\label{TD}
    dist\left(x, [F\leq\bar{F}]\right)&\leq&\|x-T_{D,\epsilon}(x)_p\|\nonumber\\
                                           &\leq&\|T_{D,\epsilon}(x)-T_{D,\epsilon}(x)_p\|+\|x-T_{D,\epsilon}(x)\|\nonumber\\
                                           &=&dist\left(T_{D,\epsilon}(x),[F\leq\bar{F}]\right)+dist\left(x,T_{D,\epsilon}(x)\right)\nonumber\\
                                           &\leq&\beta dist\left(x,[F\leq\bar{F}]\right)+dist\left(x,T_{D,\epsilon}(x)\right).\nonumber
    \end{eqnarray}
    By the statement (iv) of Proposition~\ref{prop:Gk}, we have
\[
dist\left(x,[F\leq\bar{F}]\right)\leq\frac{1}{(1-\beta)}dist\left(x,T_{D,\epsilon}(x)\right)\leq\frac{\overline{\epsilon}}{(1-\beta)(m-\overline{\epsilon}\rho)}\sqrt{\frac{\overline{\epsilon}}{\underline{\epsilon}}}dist\left(0,\partial_PF(x)\right),
\]
which shows that $F$ satisfies the strong level-set subdifferential EB on $[\bar{F}<F<\bar{F}+\nu]$. Then by Corollary \ref{Corollary:Strong} we get that $F$ satisfies the strong level-set Bregman proximal EB.
\end{proof}
\begin{remark}\label{remark4}
For problem (P), if  $F$ attains the global minimum value $F^*$ at every critical point , then solution set $\mathbf{X}^*=[F\leq F^*]$, $x^k\in[F^*<F<F^*+\nu]$, and  the inequality ~\eqref{x-Q-linear} with respect to $[F\leq F^*]$ becomes
\begin{eqnarray}\label{x-Q-linear-2}
dist\left(x^{k+1},\mathbf{X}^*\right)\leq\beta dist\left(x^k,\mathbf{X}^*\right).
\end{eqnarray}
Observe that a convex or an invex function $F$ satisfies (\ref{x-Q-linear-2}). Furthermore, conditions such as proximal-PL, a global version of K{\L} and proximal EB in~\cite{Schmidt2016} also guarantee (\ref{x-Q-linear-2}).
\end{remark}
\section{Connections with known error bounds in literature and applications }\label{sec:LSEB relationships}
This section examines the novelty of level-set error bounds and their relationships with existing error bounds. The established linear convergence results of VBPG allow us to exploit the novel convergence results for various existing algorithms. Although  we only study the ``local" version error bounds on $\mathfrak{B}(\bar{x};\eta,\nu)$ in this section, but the same analysis used in this section can be readily extended to  ``global" version error bounds on $[\bar{F}<F<\bar{F}+\nu]$.
\subsection{Error bounds with target set $\bar{\mathbf{X}}_P$ or target value $F(\bar{x})$ }\label{subsec:error bounds of point}

Let $\bar{x}\in\bar{\mathbf{X}}_P$, first we study conditions under which  the distance from any vector $x\in
\mathfrak{B}(\bar{x};\eta,\nu)$ to the set  $\bar{\mathbf{X}}_P$ is bounded by a residual function $R_1(x)$, raised to a certain power, evaluated at $x$.
Specifically, we study the existence of some $\gamma_1$, $\delta_1$, such that
$$dist^{\gamma_1}(x,\bar{\mathbf{X}}_P)\leq\delta_1 R_1(x)\;\forall x\in\mathfrak{B}(\bar{x};\eta,\nu).$$
An expression of this kind is called a first type error bound with target set $\bar{x}_P$ for (P).
\begin{definition}[First type error bounds]
\begin{itemize}
\item[(i)]({\bf Weak metric-subregularity}) We say that $\partial_P F$ is weakly metrically subregular at $\bar{x} \in \bar{X}_P$ for the zero vector $0$ if there exist $\eta$, $\nu$ and $c_2$ such that
\begin{equation}\label{subregularity}
dist\big{(}x, \bar{\mathbf{X}}_P\big{)}\leq c_2 dist\big{(}0,\partial_PF(x)\big{)}\;\forall x\in\mathfrak{B}(\bar{x};\eta,\nu).
\end{equation}
\item[(ii)]({\bf Bregman proximal error bound})
Given a Bregman function $D$ along with $\epsilon>0$, we say that the Bregman proximal error bound (EB) holds
at $\bar{x} \in \bar{X}_P$  if there exist $\eta$, $\nu$ and $c_3$ such that
\begin{equation}\label{Bregman proximal error bound}
dist(x,\bar{\mathbf{X}}_P)\leq c_3dist\big{(}x, T_{D,\epsilon}(x)\big{)}\;\forall x\in\mathfrak{B}(\bar{x};\eta,\nu).
\end{equation}
\item[(iii)]({\bf Luo-Tseng error bound~\cite{Tseng2009}})
We say the Luo-Tseng error bound (EB) holds if any $\xi\geq\inf_{x\in\RR^n} F(x)$, there exists constant $c_4>0$ and $\sigma>0$ such that
\[
dist(x,\bar{\mathbf{X}}_P)\leq c_4\|x-T_{D,\epsilon}(x)\|\quad\mbox{with}\quad D(x,y)=\frac{\|x-y\|^2}{2}
\]
whenever $F(x)\leq\xi$, $\|x-T_{D,\epsilon}(x)\|\leq\sigma$.
\end{itemize}
\end{definition}
A few remarks about (\ref{subregularity}) are in order. Metric subregularity of a set-valued mapping is a well-known notion in
variational analysis. See the monograph \cite{DoR2009} by Dontchev and Rockafellar for motivations, theory, and applications.
In (\ref{subregularity}) if $\mathfrak{B}(\bar{x};\eta,\nu)$ is replaced  by $\mathbb{B}(\bar{x};\eta)$, then
(\ref{subregularity}) is equivalent to metric subregularity of the set-value mapping $\partial_P F$ at $\bar x$ for the vector $0$ (see Exercise 3H.4 of \cite{DoR2009}) for a proof.

The following proposition provides a sufficient condition for Bregman proximal error bound.
\begin{proposition}\label{prop:LuoTseng}
The Luo-Tseng EB condition implies the Bregman proximal EB when $g$ is semiconvex.
\end{proposition}
\begin{proof}
Taking $\xi>\bar{F}$, $\nu\geq\xi-\bar{F}$, then $[\bar{F},\bar{F}+\nu]\subset[F\leq\xi]$. By Luo-Tseng EB condition, for $\xi>\bar{F}$, there are $c_4$ and $\sigma_{\xi}$ such that
$$dist(x,\bar{X}_P)\leq c_4\|x-T(x)\|,\quad F(x)\leq\xi,\quad\|x-T(x)\|\leq\sigma_{\xi}$$
Since $T(x)$ is continuous, then for $\sigma_{\xi}$, there is $\hat{\eta}_{\xi}$ such that
$$\|T(x)-T(\bar{x})\|\leq\frac{\sigma_{\xi}}{2}\qquad when\qquad\|x-\bar{x}\|\leq\hat{\eta}_{\xi}.$$
Now let $\eta=\min\{\frac{\sigma_{\xi}}{2},\hat{\eta}_{\xi}\}$, if $\|x-\bar{x}\|\leq\eta$, we have
\[
\|x-T(x)\|\leq\|x-\bar{x}\|+\|T(x)-T(\bar{x})\|=\|x-\bar{x}\|+\|T(x)-T(\bar{x})\|.
\]
Therefore $\forall x\in\mathbb{B}(\bar{x};\eta)\cap[\bar{F}<F\leq\bar{F}+\nu]$, we have $\|x-T(x)\|\leq\sigma_{\xi}$, $F(x)\leq\xi$. By Luo-Tseng EB, we have
$$dist(x,\bar{X}_P)\leq c_4\|x-T(x)\|,$$
which shows that the Bregman proximal EB holds at $\bar{x}$.
\end{proof}

We next examine the second type error bounds with target value $F(\bar{x})$. These error bounds are used to bound the absolute difference  of  any function value $F$  at $\bar{x}\in\bar{\mathbf{X}}_P$ from a test set to  the value $\bar{F}=F(\bar{x})$ by a residual function $R_2$. Specifically we study if there exist some $\gamma_2$, $\delta_2$ such that
$$R_2(x)\geq\delta_2\left(F(x)-\bar{F}\right)^{\gamma_2}\;\forall x\in\mathfrak{B}(\bar{x};\eta,\nu).$$
\begin{definition}[Second type error bounds]
\begin{itemize}
\item[(i)]{\bf Kurdyka-{\L}ojasiewicz property} The proper lower semicontinuous function $F$ is said to satisfy the Kurdyka-{\L}ojasiewicz  (K{\L}) property at $\bar{x}$ with exponent $\alpha\in(0,1)$, if there exist $\nu>0$, $\eta>0$, and $c_5>0$ such that
\[
dist(0,\partial_LF(x))\geq c_5[F(x)-\bar{F}]^{\alpha}~~~\forall x\in\mathfrak{B}(\bar{x};\eta,\nu).
\]
\item[(ii)]{\bf Bregman proximal gap condition}
Given a Bregman function $D$ along with  $\epsilon>0$, we say that the function $F$  satisfies the
Bregman proximal  (BP) gap condition relative to $D$ and $\epsilon$  at $\bar{x}\in{\rm dom}F$ with exponent $q\in[0,2)$ if there exist $\nu>0$, $\eta>0$, and $\mu>0$ such that
$$G_{D,\epsilon}(x)\geq\mu\big{(}F(x)-\bar{F}\big{)}^q \;\forall x\in\mathfrak{B}(\bar{x};\eta,\nu),$$
where $G_{D,\epsilon}(x)=-\frac{1}{\epsilon}\min\limits_{y\in\RR^n}\big{\{}\langle\nabla f(x),y-x\rangle+g(y)-g(x)+\frac{1}{\epsilon}D(x,y)\big{\}}$.
\end{itemize}
\end{definition}


\subsection{Relationships between level-set error bounds and other error bounds}
\noindent{\bf Assumption} (H):
For $\bar{x}\in\bar{\mathbf{X}}_P$, there is a $\delta>0$ such that $F(y)\leq F(\bar{x})$ whenever $y\in\bar{\mathbf{X}}_P$ and $\|y-\bar{x}\|\leq\delta$.

The next proposition will establish the relationships between level-set based EB and the first type EB.
\begin{proposition}\label{prop:levelset_critical}
Suppose that Assumption {\bf(H)} holds at $\bar{x}\in\bar{\mathbf{X}}_P$, for $x\in\mathfrak{B}(\bar{x};\eta,\nu)$ with $\eta\leq\delta$, then we have
$$dist(x,[F\leq\bar{F}])\leq dist(x,\bar{\mathbf{X}}_P).$$
\end{proposition}
\begin{proof}
By Assumption {\bf(H)}, for $\eta\leq\delta$, we have that $\bar{\mathbf{X}}_P\cap\mathfrak{B}(\bar{x};\eta,\nu)=\emptyset$. For given $x\in\mathfrak{B}(\bar{x};\eta,\nu)$, let $x_p=Proj_{\bar{\mathbf{X}}_P}(x)$, then we must have $F(\bar{x}_p)\leq\bar{F}$, $x_p\in[F\leq\bar{F}]$. Therefore, for $x\in\mathfrak{B}(\bar{x};\eta,\nu)$, we conclude the result.
\end{proof}
We are ready to present a key result on how the level-set subdifferential EB condition relates to some important notions in variational analysis and optimization.
\begin{theorem}\label{theorem:subregularity}{\bf (K{\L} property, Bregman proximal EB, and weak metric-subregularity imply level-set subdifferential EB)} For the proper l.s.c function $F$, the following assertions hold.
\begin{itemize}
\item[{\rm(a)}] Suppose that $F$ satisfies the K{\L} property at $\bar x$ with exponent $\alpha\in[0,1)$ over $\mathfrak{B}(\bar{x};\eta,\nu)$. Then $F$ satisfies the level-set subdifferential EB condition at $\bar{x}$ with $\gamma=\frac{\alpha}{1-\alpha}$ over $x\in\mathfrak{B}(\bar{x};\frac{\eta}{2},\nu)$.
So $\alpha=\frac{\gamma}{1+\gamma}$. As a consequence,
if $\alpha\in[0,1/2]$, then one has $\gamma\in(0,1]$, and if $\alpha\in (\frac{1}{2},1)$, then we have $\gamma>1$.
\item[{\rm(b)}] Suppose Assumption~\ref{assump1} holds, and Assumption {\bf(H)} holds
at $\bar{x}\in\bar{\mathbf{X}}_P$. If one of the following condition holds
\begin{itemize}
\item[{\rm(i)}] the Bregman proximal EB holds at $\bar{x}$, and $g$ is semiconvex;
\item[{\rm(ii)}] $\partial_P F$ is weakly metric-subregular at $\bar{x}$ for the zero vector $0$;
\end{itemize}
then  the level-set subdifferential EB condition holds at $\bar{x}$ with
$\gamma=1$.
\end{itemize}
\end{theorem}
\begin{proof} {\rm(a):} By Proposition 3.16 and Theorem 3.22 of \cite{kru}, we can easily get that there exists some $c_1>0$ such that
\[
dist^{\gamma}(x,[F\leq \bar{F}])\leq c_1 dist(0, \partial_L F(x))\;\forall x\in \mathbb{B}(\bar{x};\frac{\eta}{2},\nu),
\]
where $\gamma=\frac{\alpha}{1-\alpha}$. Since $\partial_P F(x)\subset \partial_L F(x)\;\forall x\in \RR^n$, the claim is proved.\\
{\rm(b):} {\rm(i):} By Proposition~\ref{prop:levelset_critical} and the Bregman proximal EB, we have
\[
dist\left(x,[F\leq\bar{F}]\right)\leq dist(x,\bar{\mathbf{X}}_P)\leq c_3 dist\big{(}x,T_{D,\epsilon}(x)\big{)}.
\]
Then by (iv) of Proposition~\ref{prop:Gk} one has
\[
dist\left(x,[F\leq\bar{F}]\right)\leq c_1 dist(0,\partial_PF(x))\;\mbox{with}\; c_1=c_3\left(\frac{\overline{\epsilon}}{m-\overline{\epsilon}\rho}\right)\sqrt{\frac{\overline{\epsilon}}{\underline{\epsilon}}}.
\]
{\rm(ii):} By Assumption {\bf(H)}, for $\eta\leq\delta$, we have that $\bar{\mathbf{X}}_P\subset[F\leq\bar{F}]$. For $x\in\mathfrak{B}(\bar{x};\eta,\nu)$, since $\partial_P F$ satisfies weak metric subregularity, we have
\[
c_2 dist\big{(}x,\partial_PF(x)\big{)}\geq dist\big{(}x, \bar{\mathbf{X}}_P\big{)}\geq dist\left(x,[F\leq\bar{F}]\right),
\]
which yields the desired result.
\end{proof}
The following proposition provides the value proximity in terms of the distance between $x$ and the set $[F\leq\bar{F}]$. Thanks of this proposition, we will establish the connection of level set EB with second type EB.
\begin{proposition}[Function-value proximity in terms of  level sets]\label{proposition1.5}
Suppose that Assumptions~\ref{assump1} and \ref{assump2} hold. If $\overline{\epsilon}<\frac{m}{L}$, then there is some $c_0=\frac{3}{2}L+\frac{M}{2\underline{\epsilon}}>0$ such that the   following estimation holds.
$$F\left(t_{D,\epsilon}(x)\right)-\overline{F}\leq E_{D,\epsilon}(x)-\overline{F}\leq c_0{dist}^2(x,[F\leq\overline{F}])\;\forall x\in [F>\overline{F}]\;\mbox{and}\; t_{D,\epsilon}(x)\in T_{D,\epsilon}(x).$$
\end{proposition}
\begin{proof} With the given choice of $\epsilon$,   $T_{D,\epsilon}(x)$ is nonempty by Proposition~\ref{prop:Ek}. So $E_{D,\epsilon}(x)$ has a finite value for any given $x$.
For $x\in [F>\bar{F}]$, let $x_p\in[F\leq\bar{F}]$ such that $\|x-x_p\|=dist(x,[F\leq\bar{F}])$. Since $F(x_p)\leq\bar{F}$, then
\begin{eqnarray}\label{eq:4.1}
F(t_{D,\epsilon}(x))-\bar{F}&\leq & E_{D,\epsilon}(x)-\bar{F}\qquad\qquad\qquad\qquad\qquad\qquad\qquad\qquad\mbox{(by (i) of Proposition~\ref{prop:Ek})}   \nonumber        \\
&\leq&\min_{y\in\RR^n}\big{\{}f(x)+\langle\nabla f(x),y-x\rangle+g(y)+\frac{1}{\epsilon}D(x,y)\big{\}}-(f+g)(x_p)\nonumber\\
&\leq&f(x)+\langle\nabla f(x),x_p-x\rangle+g(x_p)+\frac{1}{\epsilon}D(x,x_p)-(f+g)(x_p)\nonumber\\
&=&f(x)-f(x_p)+\langle\nabla f(x),x_p-x\rangle+\frac{1}{\epsilon}D(x,x_p)\nonumber\\
&\leq&\langle\nabla f(x_p)-\nabla f(x),x-x_p\rangle+\frac{L}{2}\|x-x_p\|^2+\frac{1}{\epsilon}D(x,x_p)\quad\mbox{(by Assumption~\ref{assump1})}\nonumber\\
&\leq&\frac{3}{2}L\|x-x_p\|^2+\frac{M}{2\underline{\epsilon}}\|x-x_p\|^2\qquad\qquad\qquad\qquad\quad\mbox{(by Assumptions~\ref{assump1} and~\ref{assump2})}\nonumber\\
&\leq&c_0\|x-x_p\|^2
=c_0dist^2(x,[F\leq\bar{F}]).\qquad\mbox{(where $c_0=\frac{3}{2}L+\frac{M}{2\underline{\epsilon}}$).}
\end{eqnarray}
This completes the proof.
\end{proof}
Under the assumption of semi-convexity of $g$ at $\bar{x}$, we have the following theorem, which gives
an answer to  the converse of statement (a) of Theorem~\ref{theorem:subregularity}.
\begin{proposition} \label{theo:4.3}{\bf(Level-set Bregman EB implies BP gap condition and K{\L} property)}
Suppose that Assumption 1 holds and $g$ is semi-convex.
For a given  Bregman function $D$ along with $\epsilon>0$ satisfying Assumption~\ref{assump2}, we have the following statements:
\begin{itemize}
\item[\rm{(i)}] If $F$ satisfies level-set Bregman EB holds at $\bar{x}$ with exponent $p$ over $\mathfrak{B}(\bar{x};\eta,\nu)$, then BP gap condition holds at $\bar{x}$  with exponent $q=\frac{1}{\min\{\frac{1}{p},1\}}$
over $\mathfrak{B}(\bar{x};\eta,\nu)$.
\item[\rm{(ii)}] If $F$ satisfies BP gap condition at $\bar{x}$ with exponent $q$ over $\mathfrak{B}(\bar{x};\eta,\nu)$, then function $F$ has the K{\L} property at $\bar{x}$ with exponent of $\frac{q}{2}$ over $\mathfrak{B}(\bar{x};\eta,\nu)$ .
\end{itemize}
\end{proposition}
\begin{proof}
\rm{(i):} For $x\in\mathfrak{B}(\bar{x};\eta,\nu)$, let $x_p\in[F\leq\bar{F}]$ s.t. $\|x-x_p\|=dist(x,[F\leq\bar{F}])$. We have $F(x_p)\leq F(\bar{x})=\bar{F}$ and the estimate for term $E_{D,\epsilon}(x)-\bar{F}$ can obtained by Proposition~\ref{proposition1.5} as following
\begin{equation}{\label{eq:EF}}
E_{D,\epsilon}(x)-\bar{F}\leq c_0{dist}^2(x,[F\leq\bar{F}])\quad\mbox{with}\quad c_0=\frac{3}{2}L+\frac{M}{2\underline{\epsilon}}.
\end{equation}
Furthermore, we obtain
\begin{eqnarray*}
F(x)-\bar{F}&=&F(x)-E_{D,\epsilon}(x)+E_{D,\epsilon}(x)-\bar{F}\\
                 &\leq&F(x)-E_{D,\epsilon}(x)+c_0dist^2(x,[F\leq \bar{F}])\qquad\mbox{(by~\eqref{eq:EF})}\\
                 &\leq&\epsilon G_{D,\epsilon}(x)+c_0\theta^2dist^{\frac{2}{p}}\left(x,T_{D,\epsilon}(x)\right)\;\mbox{(by level-set Bregman EB condition)}\\
                 &\leq&\epsilon G_{D,\epsilon}(x)+c_0\theta^2\left(\frac{2\overline{\epsilon}^2}{m-\overline{\epsilon}\rho}\right)^{\frac{1}{p}}\left(G_{D,\epsilon}(x)\right)^{\frac{1}{p}}.\quad\mbox{(by (ii) of Proposition~\ref{prop:Gk})}
\end{eqnarray*}
So, there is some $\mu>0$ such that
$$G_{D,\epsilon}(x)\geq\mu\left(F(x)-\bar{F}\right)^q,\quad\forall x\in\mathfrak{B}(\bar{x};\eta,\nu),\quad q=\frac{1}{\min\{\frac{1}{p},1\}}.$$
The proof is completed.\\
\rm{(ii):} By the hypothesis, the BP gap condition  holds at $\bar{x}$ with exponent $q$
over $\mathfrak{B}(\bar{x};\eta,\nu)$ , i.e.,
$$G_{D,\epsilon}(x)\geq\mu\left(F(x)-\bar{F}\right)^q~~\forall x\in \mathfrak{B}(\bar{x};\eta,\nu). $$
By the assumptions for $g$, one has $\partial_Pg(x)=\partial_Lg(x)$ and $\partial_P F(x)=\partial_L F(x)$ $\forall
x\in \mathfrak{B}(\bar{x};\eta,\nu)$. Then by (iii) of Proposition~\ref{prop:Gk} we have
\[
G_{D,\epsilon}(x)\leq\frac{\overline{\epsilon}}{2\underline{\epsilon}(m-\overline{\epsilon}\rho)}dist^2(0,\partial_LF(x)).
\]
It follows that
\[
\left(\frac{2\underline{\epsilon}}{\overline{\epsilon}}\right)(m-\overline{\epsilon}\rho)\mu\left(F(x)-\bar{F}\right)^q\leq\left[dist\big{(}0,\partial_L F(x)\big{)}\right]^2,
\]
which implies that
\[
dist\big{(}0,\partial_L F(x)\big{)}\geq\sqrt{\left(\frac{2\underline{\epsilon}}{\overline{\epsilon}}\right)(m-\overline{\epsilon}\rho)\mu}\left(F(x)-\bar{F}\right)^{\frac{q}{2}}.
\]
The assertion is justified.
\end{proof}
\subsection{Examples illustrating the novelty of the level-set subdifferential EB condition}\label{subsec:examples}
The following examples show that K{\L} property, weak metric subregularity or Bregman proximal EB are not be necessary for level-set subdifferential EB.
\begin{example}[level-set subdifferential EB does not imply weak metric subregularity] Let $\bar{x}=(0,0)^T$ and $F:\R^2\rightarrow \R$ be defined by
\[
  F(x):=\left\{
  \begin{array}{cc}
    x_1^2-x_2^3 & \mbox{if}\;x_2>0 \\[0.2cm]
    x_2^3 & \mbox{if}\; x_2\leq 0.
  \end{array}
  \right.
\]
By some direct calculations, we get that $F$ is a lower semicontinuous function, $\bar{x}$ is the unique proximal critical point,
\[
[F\leq 0]=\R\times (-\infty, 0]\cup \{(x_1,x_2)^T|x_1^2\leq x_2^3\}
\]
and
\[
\nabla F(x)=(2x_1,-3x_2^2)^T\;\mbox{when}\;x\not\in [F\leq 0].
\]
Since for each $x$ closed enough to $\bar{x}$ with $\nu >F(x)>0$ and some $\nu>0$ one has
\[
d(x,[F\leq 0])\leq |x_1|\leq \frac{1}{2}||\nabla F(x)||,
\]
$F$ satisfies the level-set subdifferential EB condition at $\bar{x}$ with the exponent $\gamma=1$. However, $F$ is not weakly metric-subregular at $\bar{x}$ for the zero vector $0$, since for $x_n=(\frac{1}{n^\frac{5}{4}},\frac{1}{n})$ which is not in $[F\leq 0]$ one has

\[
\frac{\sqrt{x_{n,1}^2+x_{n,2}^2}}{\sqrt{4x_{n,1}^2+9x_{n,2}^4}}\rightarrow \infty\;\mbox{as}\; n\to \infty.
\]
\end{example}

\begin{example}[level-set subdifferential EB does not imply K{\L} property] (See Example 3.19 of \cite{kru}) Let $F:\R\to \R$ be given by
\[
  F(x):=\left\{
  \begin{array}{cc}
    0 & \mbox{if}\;x\leq 0 \\[0.2cm]
    x^2+\frac{1}{n}-\frac{1}{n^2} & \;\;\;\mbox{if}\; \frac{1}{n}<x\leq \frac{1}{n-1}, n=3,4,\cdot\cdot\cdot,\\[0.2cm]
    x^2+\frac{1}{4} & \mbox{if}\;x>\frac{1}{2}.
  \end{array}
  \right.
\]
In Example 3.19 of \cite{kru}, the authors have obtained that $F$ is lower semicontinuous and for each $x>0$
\[
  \partial F(x)=\left\{
  \begin{array}{cc}
    [2x, \infty) & \mbox{if}\;x= \frac{1}{n-1}, n=3,4,\cdot\cdot\cdot,\\[0.2cm]
    2x & \mbox{otherwise}.
  \end{array}
  \right.
\]
Let $\bar{x}=0$ be the reference point. Then we easily get that $F$ satisfies the level-set subdifferential EB condition at $\bar{x}$ with the exponent $\gamma=1$. However, $F$ does not have K{\L} property at $\bar{x}$ with any exponent $\alpha\in [0,1)$, since for $x_n=\frac{1}{n-1}$ one has
\[
\frac{d(0,\partial F(x_n))}{F(x_n)^{\alpha}}\leq \frac{\frac{2}{n-1}}{\frac{1}{n^{\alpha}}}=\frac{2n^{\alpha}}{n-1}\rightarrow 0\;\mbox{as}\; n\to \infty.
\]
Moreover, we can verify that $F$ is not semi-convex. Indeed, for any given $\frac{1}{n}<x< \frac{1}{n-1}$ with any positive integer $n\geq 3$ and any positive constant $\rho>0$, one easily obtains that the following inequality does not hold when $x'<x$ and $x'$ sufficiently closes to $x$:
\[
\frac{\rho+2}{2}x'^2\geq \frac{\rho+2}{2}x^2+\langle 2x, x'-x\rangle,
\]
which implies that $F$ is not semi-convex by Proposition 8.12 of \cite{Rockafellar}.
\end{example}
\begin{example} [level-set subdifferential EB does not imply the Bregman proximal EB] Let $\bar{x}=(0,0)^T$ and $F:\R^2\rightarrow \R$ be defined by
\[
  F(x):=\left\{
  \begin{array}{cc}
    x_1^2-x_2^3 & \mbox{if}\;x_2>0 \\[0.2cm]
    x_2^3 & \mbox{if}\; x_2\leq 0.
  \end{array}
  \right.
\]
From Example 1.1, we get that $F$ satisfies the level-set subdifferential EB condition at $\bar{x}$ with the exponent $\gamma=1$. Moreover, it is easy to verify that Assumption ($H$) holds. By the first-order optimality condition, we can easily calculate that $T_{D,1}(x)=\{(x_1,0)^T\}$ with $D(x,y)=\frac{1}{2}||x-y||^2$ for $x\not\in [F\leq 0]$ sufficiently closing to $\bar{x}$ and $x_1>2x_2>0$. However, $F$ does not have the Bregman proximal EB at $\bar{x}$, since for $x_n=(\frac{1}{n},\frac{1}{3n^2})$ which is not in $[F\leq 0]$ one has
\[
\frac{\sqrt{x_{n,1}^2+x_{n,2}^2}}{|x_{n,2}|}\rightarrow \infty\;\mbox{as}\; n\to \infty.
\]
\end{example}

Figure 1 summarizes the main results of this section.
\begin{figure}
\centering
\begin{center}
\scriptsize
		\tikzstyle{format}=[rectangle,draw,thin,fill=white]
		\tikzstyle{test}=[diamond,aspect=2,draw,thin]
		\tikzstyle{point}=[coordinate,on grid,]
\begin{tikzpicture}
[
>=latex,
node distance=5mm,
 ract/.style={draw=blue!50, fill=blue!5,rectangle,minimum size=6mm, very thick, font=\itshape, align=center},
 racc/.style={rectangle, align=center},
 ractm/.style={draw=red!100, fill=red!5,rectangle,minimum size=6mm, very thick, font=\itshape, align=center},
 cirl/.style={draw, fill=yellow!20,circle,   minimum size=6mm, very thick, font=\itshape, align=center},
 raco/.style={draw=green!500, fill=green!5,rectangle,rounded corners=2mm,  minimum size=6mm, very thick, font=\itshape, align=center},
 hv path/.style={to path={-| (\tikztotarget)}},
 vh path/.style={to path={|- (\tikztotarget)}},
 skip loop/.style={to path={-- ++(0,#1) -| (\tikztotarget)}},
 vskip loop/.style={to path={-- ++(#1,0) |- (\tikztotarget)}}]

        \node (a) [ractm, xshift=-20]{\baselineskip=3pt\small {\bf level-set subdifferential EB}\\ \baselineskip=3pt\footnotesize$dist^{\gamma}\left(x,[F\leq\bar{F}]\right)\leq c_1{dist}\left(0,\partial_PF(x)\right)$\\
        \baselineskip=3pt\footnotesize$\forall x\in\mathfrak{B}(\bar{x};\eta,\nu)$};
        \node (b) [ract, below = of a, xshift=40]{\baselineskip=3pt\small {\bf level-set Bregman EB}\\ \baselineskip=3pt\footnotesize$dist^p\left(x,[F\leq\bar{F}]\right)\leq\theta dist\left(x,T_{D,\epsilon}(x)\right)$\\
        \baselineskip=3pt\footnotesize$\forall x\in\mathfrak{B}(\bar{x};\eta,\nu)$};
        \node (b1) [ract, right = of b, xshift=20]{\baselineskip=3pt\small {\bf BP gap condition}\\
                          \baselineskip=3pt\footnotesize$G_{D,\epsilon}\left(x\right)\geq\mu\left(F(x)-\bar{F}\right)^{q}$\\
        \baselineskip=3pt\footnotesize$\forall x\in\mathfrak{B}(\bar{x};\eta,\nu)$};
        \node (bbb) [racc, below= of b, xshift=86, yshift=28]{\baselineskip=3pt\footnotesize$g$ is\\
                                                              \baselineskip=3pt\footnotesize semi-convex};
        \node (aa1) [cirl, above = of a,yshift=-8]{$+$};
        \node (aa2) [cirl, left = of a]{$+$};
        \node (aa2a) [racc, below= of aa2, yshift=16]{\baselineskip=3pt\footnotesize $g$ is\\
                                               \baselineskip=3pt\footnotesize semi-convex};
        \node (aba) [above = of aa2, yshift=4] {\baselineskip=3pt\footnotesize {\bf(H)}};
        \node (aba1) [racc, above= of aba, xshift=-15, yshift=-2]{\baselineskip=3pt\footnotesize $g$ is semi-convex};
        \node (aa3) [ract, left = of aa2]{\baselineskip=3pt\small {\bf Bregman proximal EB}\\
        \baselineskip=3pt\footnotesize$dist\left(x,\bar{\mathbf{X}}_P\right)\leq c_3dist\left(x,T_{D,\epsilon}(x)\right)$\\
        \baselineskip=3pt\footnotesize$\forall x\in\mathfrak{B}(\bar{x};\eta,\nu)$};
        \node (aa4) [ract, above = of aa1]{\baselineskip=3pt\small {\bf weak metric subregularity}\\ \baselineskip=3pt\footnotesize$dist\left(x,\bar{\mathbf{X}}_P\right)\leq c_2 dist\left(0,\partial_PF(x)\right)$\\
        \baselineskip=3pt\footnotesize\qquad\qquad\qquad\qquad$\forall x\in\mathfrak{B}(\bar{x};\eta,\nu)$};
        \node (aa5) [ract, left = of aa4, xshift=-35, yshift=15]{\baselineskip=3pt\small {\bf metric subregularity}\\ \baselineskip=3pt\footnotesize$dist\left(x,\bar{\mathbf{X}}_P\right)\leq c_2 dist\left(0,\partial_PF(x)\right)$\\
        \baselineskip=3pt\footnotesize\qquad\qquad\qquad\qquad$\forall x\in\mathbb{B}(\bar{x},\eta)$};
        \node (dd)[ract, below = of b,xshift=-142, yshift=-55]{\baselineskip=3pt\small $\{dist(x^k,[F\leq\bar{F}])\}$\\
                                       \baselineskip=3pt\small Q-linear};
        \node (dd1) [racc, below= of dd, xshift=-95, yshift=27]{\baselineskip=3pt\footnotesize $g$ is semi-convex};
        \node (ddd)[ract, right = of dd, xshift=-16]{\baselineskip=3pt\small $\{F(x^k)\}$ Q-linear\\
                                         \baselineskip=3pt\small $\{x^k\}$ R-linear};
        \node (dddd)[racc, below = of ddd]{\baselineskip=3pt\small Linear convergence for VBPG};
        \node (ddddd)[ract, left = of b, xshift=-40]{\baselineskip=3pt\small {\bf strong level-set subdifferential EB}\\ \baselineskip=3pt\footnotesize$dist\left(x,[F\leq\bar{F}]\right)\leq c_1'dist\left(x,\partial_PF(x)\right)$\\
        \baselineskip=3pt\footnotesize$\forall x\in[\bar{F}< F<\bar{F}+\nu]$};
        \node (dddddd)[ract, below = of ddddd, xshift=11]{\baselineskip=3pt\small {\bf strong level-set Bregman EB}\\ \baselineskip=3pt\footnotesize$dist\left(x,[F\leq\bar{F}]\right)\leq\theta'dist\left(x,T_{D,\epsilon}(x)\right)$\\
        \baselineskip=3pt\footnotesize$\forall x\in[\bar{F}< F<\bar{F}+\nu]$};
        \node (g) [ract, right = of a]{\baselineskip=3pt\small {\bf K{\L} property}\\ \baselineskip=3pt\footnotesize $dist\left(0,\partial_LF(x)\right)\geq c_5\left(F(x)-\bar{F}\right)^{\alpha}$\\
        \baselineskip=3pt\footnotesize$\forall x\in\mathfrak{B}(\bar{x};\eta,\nu)$};
        \node (hh) [racc, below = of b, xshift=185, yshift=65]{\baselineskip=3pt\footnotesize $g$ is semi-convex};
        \path 
              (b) edge[->] (b1)
              (g) edge[->] (a)
            (aa3) edge[->] (aa2)
            (aa2) edge[->] (a)
            (aa4) edge[->] (aa1)
            (aa1) edge[->] (a)
            (aba) edge[->] (aa2)
            (aba) edge[->] (aa1)
             (b1) edge[->] (6,-0.65);
        \path ( 1.5, -0.65) edge[->] ( 1.5, -1.2);
        \path (-2.4, -0.65) edge[->] (-2.4, -5);
        \path ( -5, -2.5) edge[->] (-5, -3.1);
        \path ( -5, -4.4) edge[->] (-5, -5);
        \path (aa5) edge[->] (-3.0,3.25);
        \path (-6.75,0.65) edge[->, vh path] (-3,2.3);
        \path (dd) edge[->, hv path] (-8.7,-2.5);
        \path (ddddd) edge[->, hv path] (-3,-0.7);
\end{tikzpicture}
\caption{The relationships among the notions of the level-set subdifferential EB, subregularity of subdifferential, Bregman proximal EB, K{\L} property, level-set Bregman EB and Bregman gap condition}\label{fig:1}
\end{center}
\end{figure}
\subsection{Applications of level-set subdifferential error bounds}
\noindent {\bf Application 1: Linear convergence of regularized Jaccobi method}

In big data applications, the regularizer $g$ in problem (P) may have block separable structures, i.e., $g(x)=\sum\limits_{i=1}^{N}g_i(x_i)$, $x_i\in\RR_n^i$. In this setting, (P) can be specified as
\begin{equation}\label{separable}
\min_{x\in\RR^n} f(x_1,...,x_n)+\sum_{i=1}^{N}g_i(x_i)
\end{equation}
If we take $K^k(x)=\sum\limits_{i=1}^{N}f\left(R_i^k(x)\right)+\frac{c_i}{2}\|x_i-x_i^k\|^2$ and $D^k(x,y)=K^k(y)-\left[K^k(x)+\langle\nabla K^k(x),y-x\rangle\right]$, where $R_i^k\triangleq(x_1^k,...,x_{i-1}^k,x_i,x_{i+1}^k,...,x_n^k)$. Thus VBPG become a regularized Jaccobi algorithm. Recently, G. Bajac \cite{Banjac18} provided the linear convergence of regularized Jaccobi algorithm under quadratic growth condition for full convex problem~\eqref{separable}.

By Theorem~\ref{theo1}, for full nonconvex problem~\eqref{separable}, the VBPG method provides the linear convergence under the level-set subdifferential EB condition at the point $\bar{x}\in\overline{\mathbf{X}}_L$. For the convex problem~\eqref{separable}, the quadratic growth condition is equivalent to strongly level set subdifferential EB condition see Theorem 3.3 and Corollary 3.6~\cite{Lewis2018} for more details. Together with Theorem~\ref{theo:n-s}, we can show that the quadratic growth condition is also necessary for linear convergence in the sense of~\eqref{x-Q-linear-2}.

\noindent{\bf Application 2: Linear convergence under proximal-PL inequality and Bregman proximal gap}

\cite{Schmidt2016} proposes the concept of proximal-PL inequality for solving problem (P) where $F$ is invex function, $g$ is convex, i.e., there is $\mu>0$ such that the following inequality holds:
$$\frac{1}{2}D_g(x,L)\geq\mu\left(F(x)-F^*\right).$$
where $F^*$ is the global minimum value  and
$$D_g(x,\alpha)=-2\alpha\min_{y\in\RR^n}\left[\langle\nabla f(x),y-x\rangle+\frac{\alpha}{2}\|y-x\|^2+g(y)-g(x)\right],$$
which is a global version of Bregman proximal gap function with $D^k(x,y)=\frac{\|x-y\|^2}{2}$. \cite{Schmidt2016} proves the sequence $\{F(x^k)\}$ generated by PG method with a step size of $1/L$ linearly converges to $F^*$ under proximal-PL inequality. For the fully nonconvex case, Theorem~\ref{theo1}  shows the Q-linear convergence of $\{F(x^k)\}$ and the R-linear convergence of $\{x^k\}$ under the Bregman proximal gap condition, which is weaker than the proximal-PL inequality. Observe that the proximal PL inequality implies that every critical point achieves an optimum $F^*$, and the strong level-set subdifferential EB condition holds. If $g$ is semi-convex, by Theorem~\ref{theo:n-s} the proximal PL inequality is also a necessary condition for linear convergence in the sense of~\eqref{x-Q-linear-2}.

\noindent {\bf Application 3: Linear convergence under K{\L} property}

Various variable metric proximal gradient methods (VMPG) are provided in following algorithms for problem (P)
$$x^{k+1}\rightarrow\min\limits_{x}\langle\nabla f(x^k),x-x^k\rangle+g(x)+\frac{1}{2}\|x-x^k\|_{B_k}^2,$$
where $B_k$ is positive definite matrix. E. Chonzennx et. al. \cite{Chouzenous2014} proposed an inexact version of VMPG algorithm for problem (P) where $g$ is convex. And the authors also provided linear convergence of VMPG under K{\L} property with exponent $\frac{1}{2}$. Noted that VMPG is the special case of VBPG with $D^k=\frac{\|x-y\|_{B_k}^2}{2}$, Theorem~\ref{theo1} states that VMPG has the linear convergence for $\{x^k\}$ and $\{F(x^k)\}$ under level-set subdifferential EB condition. Moreover, the strong level-set subdifferential error bound condition on $[\bar{F}<F<\bar{F}+\nu]$ is necessary and sufficient for linear convergence in the sense of~\eqref{x-Q-linear}. Mention that if $g$ is semi-convex, level-set subdifferential EB condition with exponent $\gamma=1$ is equivalent to K{\L} exponent $\frac{1}{2}$ condition.

\section{Sufficient conditions for the  level-set subdifferential EB condition to hold on $\mathfrak{B}(\bar{x};\eta,\nu)$ with $\bar{x}\in\bar{\mathbf{X}}_P$}\label{sec:LSEB_sufficiential}
This section provides sufficient conditions to guarantee level-set subdifferential EB condition
at $\bar{x}\in\bar{\mathbf{X}}_P$ on $\mathfrak{B}(\bar{x};\eta,\nu)$, where $\bar{x}$ is a proximal critical point of $F=f+g$\\

First, we provide some new notions
on relaxed strong convexity of function $f$ on $\mathbb{B}(\bar{x};\eta)$.
Given $z\in\mathbb{B}(\bar{x};\eta)$, for brevity, we denote
$Proj_{\mathbb{B}(\bar{x};\eta)\cap\bar{\mathbf{X}}_P}(z)$ by $\overline{z}_p$.
The following notations can be viewed as the local version for that in H. Karimi et al's and I. Necoara et al's paper~\cite{Schmidt2016},\cite{Necoara2018} respectively.
\begin{itemize}
\item[1.] Local strong-convexity (LSC) on $\mathbb{B}(\bar{x};\eta)$:
    $$f(y)\geq f(x)+\langle\nabla f(x),y-x\rangle+\frac{\mu}{2}\|y-x\|^2,\quad\forall x,y\in\mathbb{B}(\bar{x};\eta).$$
\item[2.] Local essentially-strong-convexity at $\bar{x}_p$ (LESC) on $\mathbb{B}(\bar{x};\eta)$:
    \begin{eqnarray*}
    f(y)\geq f(x)+\langle\nabla f(x),y-x\rangle+\frac{\mu}{2}\|y-x\|^2,
    \;\forall x,y\in\mathbb{B}(\bar{x};\eta)\;\mbox{with}\;\bar{x}_p=\overline{y}_p.
    \end{eqnarray*}
\item[3.] Local weak- strong-convexity at $\bar{x}_p$ (LWSC) on $\mathbb{B}(\bar{x};\eta)$:
    \begin{eqnarray*}
    f(\bar{x}_p)\geq f(x)+\langle\nabla f(x),\bar{x}_p-x\rangle+\frac{\mu}{2}\|\bar{x}_p-x\|^2,
    \quad\forall x\in\mathbb{B}(\bar{x};\eta).
    \end{eqnarray*}
\item[4.] Local quadratic-gradient-growth (LQGG) at $\bar{x}_p$ on $\mathbb{B}(\bar{x};\eta)$:
    \begin{eqnarray*}
    \langle\nabla f(x)-\nabla f(\bar{x}_p), x-\bar{x}_p\rangle\geq\mu\|\bar{x}_p-x\|^2,
    \quad\forall x\in\mathbb{B}(\bar{x};\eta).
    \end{eqnarray*}
\end{itemize}
For the case $g=0$, the following two notions are introduced.
\begin{itemize}
\item[5.] Local restricted secant inequality (LRSI):
\[
\langle\nabla f(x),x-\bar{x}_p\rangle\geq\mu\|x-\bar{x}_p\|^2,\quad\forall x\in\mathbb{B}(\bar{x};\eta).
\]
\item[6.] Local Polyak-{\L}ojasiewicz (LPL) inequality:
    \[
    \frac{1}{2}\|\nabla f(x)\|^2\geq\mu\left(f(x)-f(\bar{x})\right),\quad\forall x\in\mathbb{B}(\bar{x};\eta).
    \]
\end{itemize}
It's easy to show that the following implications hold for the function $f$ on $\mathbb{B}(\bar{x};\eta)$.
$$(LSC)\Rightarrow(LESC)\Rightarrow(LWSC).$$
For the case $g=0$, the LQGG  reduce to the local restricted secant inequality (LRSI). So  we have:
$$(LWSC)\Rightarrow(LRSI)\Rightarrow(LPL)\qquad\mbox{(if $g=0$)}.$$
Along with Assumption 3, the following proposition allow us to establish the level-set subdifferential EB of $F$, when $g$ is uniformly prox-regular.
\begin{proposition}[Sufficient conditions for weak metric subregularity]\label{prop:5.1}
Suppose $\bar{x}\in\bar{x}_P$, $g$ is uniformly prox-regular around $\bar x\in \mathbf{dom}~g$ with modulus $\rho$. If  one of the following conditions holds
\begin{itemize}
\item[{\rm(i)}] $f$ is local weak strongly convex (LWSC) at $\bar{x}_p$ with modulus $\mu$ and $\mu>\rho$ on $\mathbb{B}(\bar{x};\eta)$.
\item[{\rm(ii)}] $f$ satisfies local quadratic gradient growth condition (LQGG) at $\bar{x}_p$ with modulus $\mu$ and $\mu>\rho$ on $\mathbb{B}(\bar{x};\eta)$,
\end{itemize}
then $F$ satisfies the weak metric subregularity condition at $\bar{x}$.
\end{proposition}
\begin{proof}
{\rm(i):} If $f$ is LWSC at $\bar{x}_p$ on $\mathbb{B}(\bar{x};\eta)$, then we have
\begin{equation}\label{eq:s-1}
f(\bar{x}_p)\geq f(x)+\langle\nabla f(x),\bar{x}_p-x\rangle+\frac{\mu}{2}\|\bar{x}_p-x\|^2.
\end{equation}
Since $g$ is uniformly prox-regular around $\bar x$ with $\rho$, then $\partial_Pg(x)=\partial_Lg(x)$ and
\begin{equation}\label{eq:s-2}
g(\bar{x}_p)\geq g(x)+\langle\xi,\bar{x}_p-x\rangle-\frac{\rho}{2}\|\bar{x}_p-x\|^2,\quad\forall \xi\in\partial_P g(x).
\end{equation}
Adding inequalities~\eqref{eq:s-1} and~\eqref{eq:s-2}, we obtain
$$F(\bar{x}_p)=F(\bar{x})=F_{\zeta}\geq F(x)+\langle\nabla f(x)+\xi,\bar{x}_p-x\rangle+\frac{(\mu-\rho )}{2}\|\bar{x}_p-x\|^2.$$
and
$$\langle\nabla f(x)+\xi,x-\bar{x}_p\rangle\geq\frac{(\mu-\rho )}{2}\|\bar{x}_p-x\|^2,\quad\forall\xi\in\partial_P g(x),\quad\forall x\in\mathfrak{B}(\bar{x};\eta,\nu).$$
Using Cauchy-Schwartz on above inequality, we conclude
$$dist(0,\partial_PF(x))\geq\frac{(\mu-\rho )}{2}\|\bar{x}_p-x\|\geq\frac{(\mu-\rho )}{2}dist(x,\bar{\mathbf{X}}_P),\quad\forall x\in\mathfrak{B}(\bar{x};\eta,\nu),$$
which yields the desired results.\\
{\rm(ii):} If $f$ is LQGG at $\bar{x}_p$ on $\mathbb{B}(\bar{x};\eta)$, then we have
$$\langle\nabla f(x)-\nabla f(\bar{x}_p),x-\bar{x}_p\rangle\geq\mu\|\bar{x}_p-x\|^2,\quad\forall x\in\mathfrak{B}(\bar{x};\eta,\nu).$$
Since $g$ is semi-convex, we have
$$\langle u-v,x-\bar{x}_p\rangle\geq-\rho\|x-\bar{x}_p\|^2,\quad\forall u\in\partial_Pg(x),\quad\forall v\in\partial_Pg(\bar{x}_p).$$
Adding the above two inequalities for $x\in\mathbb{B}(\bar{x};\eta)$, we obtain
$$\langle(\nabla f(x)+u)-(\nabla f(\bar{x}_p)+v),x-\bar{x}_p\rangle\geq\frac{\mu}{2}\|\bar{x}_p-x\|^2.$$
Since $\bar{x}_p$ is a proximal critical point,  $0=\nabla f(\bar{x}_p)+v$ for some $v\in\partial_P g(\bar{x}_p)$.
With this choice of $v$, the last above inequality yields
$$\langle\nabla f(x)+u,x-\bar{x}_p\rangle\geq\frac{(\mu-\rho )}{2}\|\bar{x}_p-x\|^2\geq\frac{(\mu-\rho )}{2}dist(x,\bar{\mathbf{X}}_P),\quad\forall u\in\partial_P g(x),\quad\forall x\in\mathfrak{B}(\bar{x};\eta,\nu).$$
This is enough for the proof of proposition.
\end{proof}
Now we are ready to present a main result on sufficient conditions to guarantee that
 the level-set subdifferential EB holds at $\bar{x}$ on $\mathfrak{B}(\bar{x},\eta,\nu)$, where $\bar{x}$ is an accumulation point of the sequence $\{x^k\}$ generated by VBPG.
\begin{theorem}[Sufficient conditions for the existence of a level-set subdifferential EB]
Consider problem (P). Suppose that Assumption~\ref{assump1} and Assumption~\ref{assump2} hold, and $\bar{x}\in\bar{\mathbf{X}}_P$. If one of following conditions hold, then $F$ satisfies the level-set subdifferential EB condition at $\bar{x}$ on $\mathfrak{B}(\bar{x};\eta,\nu)$.
\begin{itemize}
\item[{\rm(i)}] $F=f+g$ satisfies the K{\L} exponent at $\bar{x}$ on $\mathfrak{B}(\bar{x};\eta,\nu)$ at $\bar{x}$.
\item[{\rm(ii)}] $F=f+g$ satisfies Bregman proximal EB condition, Assumption {\bf(H)} holds, $g$ is semi-convex or $g$ is uniformly prox-regular around $\bar{x}$, $x\in\mathfrak{B}(\bar{x};\frac{\eta}{2},\frac{\nu}{N})$ with $N\geq\frac{2\overline{\epsilon}\nu}{m-\overline{\epsilon}L}/\left(\frac{\eta}{2}\right)^2$ satisfies Property (A).
\item[{\rm(iii)}] $F=f+g$ satisfies weak metric subregularity at $\bar{x}$ and Assumption {\bf(H)} holds.
\item[{\rm(iv)}] With $g=0$, $f=F$ satisfies the (LPL) inequality on $\mathbb{B}(\bar{x};\eta)$.
\end{itemize}
\end{theorem}
\begin{proof}
(i)-(iii) see Theorem~\ref{theorem:subregularity}. (iv) For this case, the (LPL) inequality implies the K{\L} property. Then the assertion follows from Proposition~\ref{theorem:subregularity}.
\end{proof}
\begin{remark}
 For the optimization problem (P),
if we consider the global solution $\mathbf{X}^*$ instead of  $\bar{\mathbf{X}}_P$, then Assumption (H) is automaticcally satisfied.  Weak metric subregularity and the Bregman proximal EB imply the level-set subdifferential EB.
\end{remark}
\begin{remark}
From the definition of a level-set subdifferential EB, suppose that $\bar{x}$ is a critical point. If $x\in\mathfrak{B}(\bar{x};\eta,\nu)$ is also a critical point, then $0\in\partial_PF(x)$ and $dist\left(x,[F\leq\bar{F}]\right)=0$. This fact follows $F(x)\leq F(\bar{x})$, which implies Assumption (H) is a necessary condition for a level-set subdifferential EB to hold. We mention that Assumption (H) is also necessary for K{\L} property.
\end{remark}

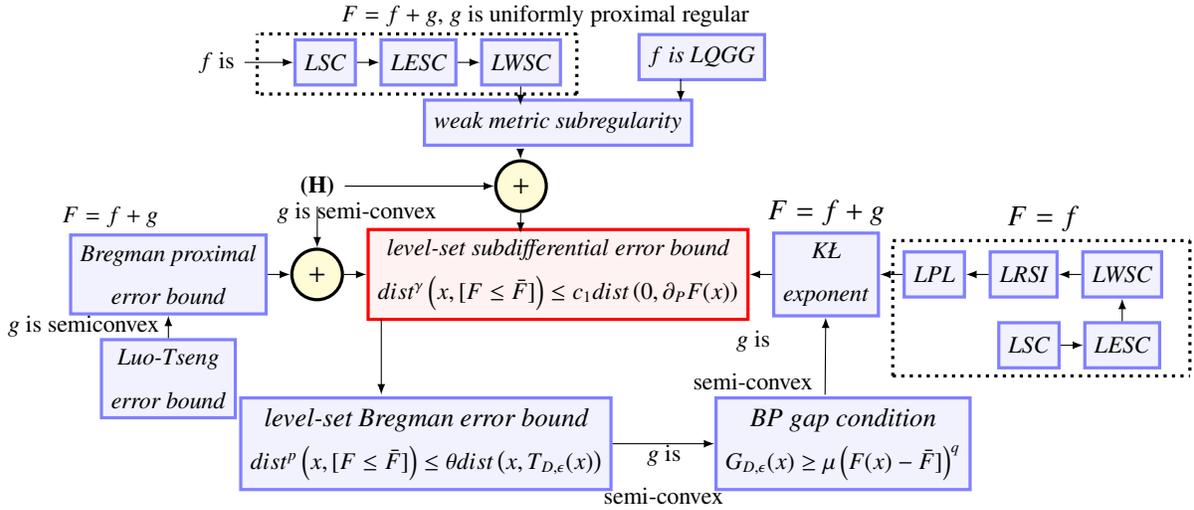
\begin{figure}
\begin{center}
\begin{tikzpicture}
[
>=latex,
node distance=3mm,
 ract/.style={draw=blue!50, fill=blue!5,rectangle,minimum size=6mm, very thick, font=\itshape, align=center},
 racc/.style={rectangle, align=center},
 ractm/.style={draw=red!100, fill=red!5,rectangle,minimum size=6mm, very thick, font=\itshape, align=center},
 cirl/.style={draw, fill=yellow!20,circle,   minimum size=6mm, very thick, font=\itshape, align=center},
 raco/.style={draw=green!500, fill=green!5,rectangle,rounded corners=2mm,  minimum size=6mm, very thick, font=\itshape, align=center},
 hv path/.style={to path={-| (\tikztotarget)}},
 vh path/.style={to path={|- (\tikztotarget)}},
 skip loop/.style={to path={-- ++(0,#1) -| (\tikztotarget)}},
 vskip loop/.style={to path={-- ++(#1,0) |- (\tikztotarget)}}]

        \node (a) [ractm]{\baselineskip=3pt\footnotesize level-set subdifferential error bound\\ \baselineskip=3pt\footnotesize$dist^{\gamma}\left(x,[F\leq\bar{F}]\right)\leq c_1{dist}\left(0,\partial_PF(x)\right)$};
        \node (b) [ract, below = of a, xshift=-50,yshift=-20]{\baselineskip=3pt\small level-set Bregman error bound\\ \baselineskip=3pt\footnotesize$dist^{p}\left(x,[F\leq\bar{F}]\right)\leq\theta dist\left(x,T_{D,\epsilon}(x)\right)$};
        \node (bb) [ract, right = of b, xshift=30]{\baselineskip=3pt\small BP gap condition\\ \baselineskip=3pt\footnotesize$G_{D,\epsilon}(x)\geq\mu\left(F(x)-\bar{F}]\right)^q$};
        \node (bbb) [racc, below= of b, xshift=90, yshift=30]{\baselineskip=3pt\footnotesize$g$ is\\
                                                              \baselineskip=3pt\footnotesize semi-convex};
        \node (bb1) [racc, above= of bb, xshift=-34,yshift=-10]{\baselineskip=0.1pt\footnotesize$g$ is\\
                                                     \baselineskip=0.1pt\footnotesize semi-convex};
        \node (d) [ract, right = of a]{\baselineskip=3pt\footnotesize K{\L}\\ \baselineskip=3pt\footnotesize exponent};
        \node (ddd) [above = of d, yshift=-10]{$F=f+g$};
        \node (e) [ract, right = of d]{\baselineskip=3pt\footnotesize LPL};
        \node (f) [ract, right = of e]{\baselineskip=3pt\footnotesize LRSI};
        \node (g) [ract, right = of f]{\baselineskip=3pt\footnotesize LWSC};
        \node (gg) [above = of g, xshift=-30, yshift=-5]{$F=f$};
        \node (h) [ract, below = of g]{\baselineskip=3pt\footnotesize LESC};
        \node (i) [ract, left = of h]{\baselineskip=3pt\footnotesize LSC};
        \node (j) [ract, above = of a, xshift=-14, yshift=45]{\baselineskip=3pt\footnotesize LWSC};
        \node (jj) [ract, below = of j, xshift=14, yshift=4]{\baselineskip=3pt\footnotesize weak metric subregularity};
        \node (aa) [cirl, below = of j, yshift=-19]{$+$};
        \node (aba) [left = of aa, xshift=-48] {\baselineskip=3pt\footnotesize {\bf(H)}};
        \node (aba1) [racc, below= of aba, xshift=15, yshift=14]{\baselineskip=3pt\footnotesize$g$ is semi-convex};
        \node (k) [ract, left = of j]{\baselineskip=3pt\footnotesize LESC};
        \node (kk) [above= of k, xshift=48, yshift=-8]{\baselineskip=3pt\footnotesize $F=f+g$, $g$ is uniformly proximal regular};
        \node (l) [ract, left = of k]{\baselineskip=3pt\footnotesize LSC};
        \node (ll) [racc, left= of l, xshift=-10]{\baselineskip=3pt\footnotesize $f$ is};
        \node (m) [ract, right = of j, xshift=20,yshift=2]{\baselineskip=3pt\footnotesize $f$ is LQGG};
        \node (n) [cirl, left = of a]{$+$};
        \node (o) [ract, left  = of n]{\baselineskip=3pt\footnotesize Bregman proximal\\
                                       \baselineskip=3pt\footnotesize error bound};
        \node (o1) [ract, below = of o]{\baselineskip=3pt\footnotesize Luo-Tseng\\
                                        \baselineskip=3pt\footnotesize error bound};
        \node (o1o) [racc, above= of o1, xshift=-32, yshift=-12]{\baselineskip=3pt\footnotesize $g$ is semiconvex};
        \node (oo) [above = of o, xshift=-22, yshift=-10]{\baselineskip=3pt\footnotesize $F=f+g$};
        \path (-2.35,-0.6) edge[->] (-2.35,-1.6)
              (3.55,-1.6) edge[->] (d)
              (b) edge[->] (bb)
              (aba) edge[->] (aa)
              (aa) edge[->] (-0.5,0.55)
              (j) edge[->] (-0.5,2.25)
              (-0.5,1.6) edge[->] (aa)
              (d) edge[->] (a)
              (e) edge[->] (d)
              (f) edge[->] (e)
              (g) edge[->] (f)
              (h) edge[->] (g)
              (i) edge[->] (h)
              (k) edge[->] (j)
              (l) edge[->] (k)
              (n) edge[->] (a)
              (aba) edge[->] (n)
              (o) edge[->] (n)
              (o1) edge[->] (o);
        \path (1.63,2.6) edge[->] (1.63,2.3)
              (ll) edge[->] (l);

        \draw[dotted,very thick] (4.45,-1.35)--(8.4,-1.35);
        \draw[dotted,very thick] (4.45,0.45)--(8.4,0.45);
        \draw[dotted,very thick] (4.45,-1.35)--(4.45,0.45);
        \draw[dotted,very thick] (8.4,0.45)--(8.4,-1.35);

        \draw[dotted,very thick] (0.25,2.4)--(-4,2.4);
        \draw[dotted,very thick] (0.25,3.2)--(-4,3.2);
        \draw[dotted,very thick] (0.25,2.4)--(0.25,3.2);
        \draw[dotted,very thick] (-4,2.4)--(-4,3.2);
\end{tikzpicture}
\caption{Sufficient conditions for the level-set subdifferential error bound}\label{fig:2}
\end{center}
\end{figure}


\begin{thebibliography} {99}
\bibitem{Attouch13}
Attouch, H., Bolte, J. and Svaiter, B. F. (2013). Convergence of descent methods for semi-algebraic and tame problems: proximal algorithms, forward-backward splitting, and regularized Gauss-Seidel methods. {\sl Mathematical Programming, 137}(1-2), 91-129.
\bibitem{Aybat2014}
Aybat, N. S. and Iyengar, G. (2014). A unified approach for minimizing composite norms. {\sl Mathematical Programming, 144}(1-2), 181-226.
\bibitem{Aze2017}
Az\'e, D. and Corvellec, J. N. (2017). Nonlinear error bounds via a change of function. {\sl Journal of Optimization Theory and Applications, 172}(1), 9-32.
\bibitem{Banjac18}
Banjac, G., Margellos K. and Goulart P. J. (2018). On the convergence of a regularized Jacobi algorithm for convex optimization. {\sl IEEE Tranations on Automatic control, 63}(4), 1113-1119.
\bibitem{Bernard05}
Bernard, F. and Thibault, L. (2005). Uniform prox-regularity of functions and epigraphs in Hilbert spaces. {\sl Nonlinear analysis, 60}, 187-207.
\bibitem{Bolte2010}
Bolte, J., Daniilidis, A., Ley, O. and Mazet, L. (2010). Characterizations of {\L}ojasiewicz inequalities: subgradient flows, talweg, convexity. {\sl Transactions of the American Mathematical Society, 362}(6), 3319-3363.
\bibitem{Bonettini2016}
Bonettini, S., Loris, I., Porta, F. and Prato, M. (2016). Variable metric inexact line-search-based methods for nonsmooth optimization. {\sl SIAM Journal on Optimization, 26}(2), 891-921.
\bibitem{BuD02} Burke J.~V. and Deng S. (2002). Weak sharp minima revisited, part I: Basic theory. {\em Control and Cybernetics 31}, 439-469.

\bibitem{BuF93}
Burke J.~V. and Ferris M.~C. (1993).
\newblock Weak sharp minima in mathematical programming.
\newblock {\em SIAM Journal on Control and Optimization 31}, 1340-1359.

\bibitem{CandesTao05}
Cand\'es, E. J. and Tao, T. (2005). Decoding by linear programming. {\sl IEEE transactions on information theory, 51}(12), 4203-4215.
\bibitem{Cohen17}
Carpentier, P. and Cohen, G. (2017). {\sl D{\'e}composition-coordination en optimisation d{\'e}terministe et stochastique}. Springer Berlin Heidelberg.
\bibitem{Chouzenous2014}
Chouzenoux, E., Pesquet, J. C. and Repetti, A. (2014). Variable metric forwardbackward algorithm for minimizing the sum of a differentiable function and a convex function. {\sl Journal of Optimization Theory and Applications, 162}(1), 107-132.
\bibitem{Cohen80}
Cohen, G. (1980). Auxiliary problem principle and decomposition of optimization problems. {\sl Journal of Optimization Theory and Applications, 32}(3), 277-305.
\bibitem{CohenZ}
Cohen, G. and Zhu D. (1984). Decomposition and coordination methods in large scale optimization problems: The nondifferentiable case and the use of augmented Lagrangians. {\sl Advances in Large Scale Systems, 1}, 203-266.
\bibitem{Cro78}
Cromme. L. (1978).
\newblock Strong Uniqueness.
\newblock {\em Numerische Mathematik, 29}, 179-193.
\bibitem{Donoho06}
Donoho, D. L. (2006). Compressed sensing. {\sl IEEE Transactions on Information Theory, 52}(4), 1289-1306.
\bibitem{DoR2009}
Dontchev, A. and Rockafellar, R.T. (2009). {\sl Implicit Functions and Solution Mappings}. Springer Science \& Business Media.

\bibitem{Lewis2018}
Drusvyatskiy, D. and Lewis, A. S. (2018). Error bounds, quadratic growth, and linear convergence of proximal methods. {\sl Mathematics of Operations Research, 43}(3), 919-948.
\bibitem{Frankel2015}
Frankel, P., Garrigos, G. and Peypouquet, J. (2015). Splitting methods with variable metric for Kurdyka-{\L}ojasiewicz functions and general convergence rates. {\sl Journal of Optimization Theory and Applications, 165}(3), 874-900.
\bibitem{Hoffman1952}
Hoffman, A. J. (1952). On approximate solutions of systems of linear inequalities. {\sl Journal of Research of the National Bureau of Standards, 49}, 263-265.
\bibitem{Schmidt2016}
Karimi, H., Nutini, J. and Schmidt, M. (2016). Linear convergence of gradient and proximal-gradient methods under the Polyak-{\L}ojasiewicz condition. In {\sl Joint European Conference on Machine Learning and Knowledge Discovery in Databases} (pp. 795-811). Springer, Cham.
\bibitem{kru}
Kruger Alexander Y, Lopez Marco A., Yang X.Q. and Zhu J. X. (2019). Holder error bounds and Holder calmness with applications ?¡ì
to convex semi-infinite optimization, {\sl Set-Valued and Variational Analysis, 27}, 995-1023.
\bibitem{LiPong18}
Li, G. and Pong, T. K. (2018). Calculus of the exponent of Kurdyka-{\L}ojasiewicz inequality and its applications to linear convergence of first-order methods. {\sl Foundations of Computational Mathematics, 18}(5), 1199-1232.
\bibitem{LionsMercier1979}
Lions, P. L. and Mercier, B. (1979). Splitting algorithms for the sum of two nonlinear operators. {\sl SIAM Journal on Numerical Analysis, 16}(6), 964-979.

\bibitem{Loj63}
Lojasiewicz. S. (1963).
\newblock  A topological property of real analytic subsets (in French).
\newblock {\em Coll du. CNRS. Les Equations
Aux derivees Partielles},
87-89.

\bibitem{LuoTseng92}
Luo, Z. Q. and Tseng, P. (1992). Error bound and convergence analysis of matrix splitting algorithms for the affine variational inequality problem. {\sl SIAM Journal on Optimization, 2}(1), 43-54.
\bibitem{Aybat2018}
Ma, S. and Aybat, N. S. (2018). Efficient optimization algorithms for robust principal component analysis and its variants. {\sl Proceedings of the IEEE, 106}(8), 1411-1426.
\bibitem{Mordukhovich}
Mordukhovich, B. S. (2006). {\sl Variational analysis and generalized differentiation I: Basic theory (Vol. 330)}. Springer Science \& Business Media.
\bibitem{Necoara2018}
Necoara, I., Nesterov, Y. and Glineur, F. (2018). Linear convergence of first order methods for non-strongly convex optimization. {\sl Mathematical Programming}, 1-39.
\bibitem{Nesterov13}
Nesterov, Y. (2013). {\sl Introductory lectures on convex optimization: A basic course} (Vol. 87). Springer Science \& Business Media.
\bibitem{Ortega}
Ortega, J. M. and Rheinboldt, W. C. (1970). {\sl Iterative solution of nonlinear equations in several variables} (Vol. 30). SIAM.

\bibitem{Pan97}
Pang. J.~S. (1997).
\newblock Error bounds in mathematical programming.
\newblock {\em Mathematical Programming, 79}, 299-332.

\bibitem{Pol63}
Polyak. B. T. (1963).
\newblock Gradient methods for minimizing functionals (in Russian).
\newblock {\em  Zhurnal Vychislitel'no{\i} Matematikii Matematichesko{\i}
Fiziki,} 643-653.

\bibitem{Rob81}
Robinson, S. M. (1980).
\newblock Some continuity properties of polyhedral multifunctions.
\newblock {\em Mathematics of Operations Research, 5}, 206-214.
\bibitem{Rockafellar}
Rockafellar, R. T. and Wets, R. J. B. (2009). {\sl Variational analysis} (Vol. 317). Springer Science \& Business Media.

\bibitem{Tibshirani1996}
Tibshirani, R. (1996). Regression shrinkage and selection via the lasso. {\sl Journal of the Royal Statistical Society, Series B (Methodological)}, 267-288.
\bibitem{Tseng2009}
Tseng, P. and Yun, S. (2009). A coordinate gradient descent method for nonsmooth separable minimization. {\sl Mathematical Programming, 117}(1-2), 387-423.
\bibitem{Ye18}
Wang, X., Ye, J., Yuan, X., Zeng, S. and Zhang, J. (2018). Perturbation techniques for convergence analysis of proximal gradient method and other first-order algorithms via variational analysis. {\sl arXiv preprint arXiv:1810.10051}.
\bibitem{Zhang2019}
Zhang, H. (2020). New analysis of linear convergence of gradient-type methods via unifying error bound conditions. {\sl Mathematical programming, 180}, 371-416.
\bibitem{Zhu94}
Zhu, D. and Marcotte, P. (1994). An extended descent framework for variational inequalities. {\sl Journal of Optimization Theory and Applications, 80}(2), 349-366.
\bibitem{ZhuMarcotte95}
Zhu, D. and Marcotte, P. (1995). Coupling the auxiliary problem principle with descent methods of pseudoconvex programming. {\sl European Journal of Operational Research, 83}(3), 670-685.
\bibitem{ZhuDeng19}
Zhu, D. and Deng, S. (2019). A Variational Approach on Level sets and Linear Convergence of Variable Bregman Proximal Gradient Method for Nonconvex Optimization Problems. {\sl arXiv preprint arXiv:1905.08445}.
\end{thebibliography}
\end{document}